\documentclass[12pt]{article}
\usepackage{amsmath,amsfonts,amssymb,amsthm}
\usepackage{mathrsfs,hyperref}
\usepackage{color}
\usepackage[square,numbers]{natbib}

\newtheorem{proposition}{Proposition}[section]
\newtheorem{theorem}[proposition]{Theorem}
\newtheorem{corollary}[proposition]{Corollary}
\newtheorem{lemma}[proposition]{Lemma}
\newtheorem{remark}[proposition]{Remark}

\numberwithin{equation}{section}
\newcommand{\nc}{\newcommand}
\nc{\R}{{\mathbb R}}
\nc{\bS}{{\mathbb S}^{d-1}}
\nc{\N}{{\mathbb N}}
\nc{\Z}{{\mathbb Z}}
\nc{\BP}{\mathbb{P}}
\nc{\BE}{\mathbb{E}}
\nc{\BQ}{\mathbb{Q}}
\nc{\BM}{\mathbb{M}}
\nc{\bN}{{\mathbf N}}
\nc{\BX}{{\mathbb X}}
\nc{\BY}{{\mathbb Y}}
\nc{\cH}{{\mathcal H}}
\nc{\cB}{{\mathcal B}}
\nc{\cX}{{\mathcal X}}
\nc{\cN}{{\mathcal N}}
\nc{\cY}{{\mathcal Y}}
\nc{\dint}{{\rm d}}

\makeatletter
\def\keywords{\xdef\@thefnmark{}\@footnotetext}
\makeatother

\title{Poisson approximation  with applications to  stochastic geometry}
\author{
	
	Federico Pianoforte\footnotemark[1] \, and \,  Matthias Schulte\footnotemark[2] }

\begin{document}
	
\footnotetext[1]{federico.pianoforte@stat.unibe.ch;  University of Bern, Institute of Mathematical Statistics and Actuarial Science.}
\footnotetext[2]{matthias.schulte@tuhh.de;  Hamburg University of Technology,  Institute of Mathematics.}
\maketitle	


\begin{abstract}
This article compares the distributions of integer-valued random variables and Poisson random variables. 
It considers the total variation and the Wasserstein distance and provides, in particular, explicit bounds on the pointwise difference between the cumulative distribution functions. Special attention is dedicated to estimating the difference when the cumulative distribution functions are evaluated at 0. This permits to approximate the minimum (or maximum) of a collection of random variables by a suitable random variable in the Kolmogorov distance. The main theoretical results are obtained by  combining the Chen-Stein method with size-bias coupling and a generalization of size-bias coupling for integer-valued random variables developed herein. A wide variety of applications are then discussed with a focus on stochastic geometry.
\end{abstract}

\section{Introduction and main results} 
Let $X$ be a random variable taking values in $\N_0 = \N \cup \{0\}$ and let $P_\lambda$ be a Poisson  random variable with mean $\lambda > 0$. In this article we employ Stein's method, size-bias coupling and a generalization of size-bias coupling for integer-valued random variables developed herein to compare the distributions of $X$ and $P_\lambda$. We derive upper bounds on the total variation distance 
\begin{align*}
d_{TV}( X , P_\lambda )
 =  \underset{A\subset \N_0}{\operatorname{sup} }\,\vert  \mathbb{P}( X \in A) - \mathbb{P}(P_\lambda \in A) \vert
\end{align*}
and the Wasserstein distance
\begin{align*}
d_{W}( X , P_\lambda )
 =  \underset{g\in\mathrm{Lip}(1)}{\operatorname{sup} }\,\vert \mathbb{E}[g(X)] - \mathbb{E}[g(P_{\lambda}) ] \vert
\end{align*}
between $X$ and $P_\lambda $, where $\mathrm{Lip}(1)$ denotes the set of all Lipschitz functions $g:\N_0\rightarrow \R$ with Lipschitz constant bounded by $1$. In addition, we establish bounds on the pointwise differences 
\begin{align*}
\big\vert \mathbb{P} (X\leq v) - \mathbb{P}(P_\lambda \leq v) \big\vert ,\quad v\in\N_0,
\end{align*}
between the cumulative distribution functions of $X$ and $P_\lambda$, which are smaller than those for the total variation distance. Particular attention is paid to the case $v = 0$. This permits to approximate the minimum (or maximum) of a collection of random variables by a suitable random variable in the Kolmogorov distance. For example, let $\lambda_d$ denote the Lebesgue measure on $\R^d$, let $k_d$ stand for the volume of the $d$-dimensional unit ball, and let  $\eta_t$ be a Poisson process on $\R^d$ with intensity measure $t\lambda_d$, $t>0$. From the aforementioned bounds for $v=0$ we deduce that the random variable $Y_t$ given by
$$
Y_t 
= \underset{(x,y)\in \eta_{t,\neq}^{2} \, : \, \frac{x+y}{2}\in [0,1]^d}{\mathrm{min}}\,  2^{-1} t^2 k_d \Vert x-y \Vert^d ,
$$
which is the rescaled  minimum (Euclidean) distance between pairs of points of $\eta_t$ with midpoint in $[0,1]^d$, satisfies
\begin{equation}\label{eq:InterpointDistances}
0 
\leq \mathbb{P}(Y_t>u) - \mathbb{P}(E_1>u)
\leq \frac{81}{t}
\end{equation} 
for  $u\geq0$ (see Theorem \ref{Thm:min interp dist}), where $E_1$ denotes an exponentially distributed random variable with mean $1$. This is possible because $\mathbb{P}(Y_t >u)$ can be written as $\mathbb{P}(X_u=0)$ with
$$
X_u = \frac{1}{2} \sum_{(x,y)\in\eta_{t,\neq}^{2}} \mathbf{1}\Big\{\frac{x+y}{2}\in [0,1]^d , 2^{-1} t^2 k_d \Vert x-y \Vert^d \in [0,u]\Big\} 
$$
and $\mathbb{P}( E_1 > u ) = \mathbb{P}(P_u =0)$. By estimating $\vert \mathbb{P}(X_u=0) - \mathbb{P}(P_u =0) \vert$ uniformly for all $u\geq 0$, one obtains \eqref{eq:InterpointDistances}, which provides a bound on the Kolmogorov distance 
$$
d_{K}(Y_t, E_1)
= \underset{u\in\R}{\operatorname{sup} }\,\vert  \mathbb{P}(Y_t > u ) - \mathbb{P}(E_1 > u) \vert
$$
between $Y_t$ and $E_1$. 

Let us now give precise formulations of our main results.  We use the shorthand notation $a\wedge b=\min\{a,b\}$ for $a,b\in\R$, and we indicate by $W_+$ and $W_-$ the positive and negative part of a random variable $W$, respectively. Whenever we write $\alpha>0$, it is understood that $\alpha\in (0,\infty)$.
\begin{theorem}\label{Thmmm}
	Let $X$ be a random variable taking values in $\mathbb{N}_0$ and let $P_{\lambda}$ be a Poisson random variable with mean $\lambda=\mathbb{E}[X]>0$. Assume there exists a random variable $Z$ defined on the same probability space as $X$ with values in $\Z$ such that 
	\begin{align}\label{Mek}
	i\mathbb{P}(X=i)  
	= \lambda \mathbb{P}(X + Z = i-1), \quad i\in\N,
	\end{align}
	is satisfied. Then, 
	\begin{align}\label{fsxx} 
	& d_{TV}(X,P_\lambda) 
	\leq (1 \wedge \lambda ) \mathbb{E}[ \vert Z \vert ] \quad \text{ and } \quad  d_{W}(X,P_\lambda) 
	\leq (1.1437 \sqrt{\lambda} \wedge \lambda )\mathbb{E}[ \vert Z \vert ].
	\end{align}
	Furthermore for all $m\in\mathbb{N}_0$,
	\begin{align}\label{rrrf} 
	\vert \mathbb{P}(X=0) - \mathbb{P}(P_\lambda =0) \vert 
	\leq \frac{m!}{\lambda^{m}}\mathbb{E}[ \vert Z \vert ] + \sum_{k=0}^{m-1} \left( \frac{\lambda}{k+1} \wedge \frac{k!}{\lambda^{k}} \right) \mathbb{E}\big[ \vert Z \vert \mathbf{1}\{ X- Z_{-} = k \} \big] 
	\end{align}
	and for all $v\in \N$,
	\begin{align}\label{eq:boundOnv} 
	\vert \mathbb{P}(X\leq v) - \mathbb{P}(P_\lambda \leq v) \vert 
	\leq \frac{(v+1)^2}{\lambda}\mathbb{E}[ \vert Z \vert ] +  \mathbb{E}\big[ \vert Z \vert \mathbf{1}\{ X- Z_{-} \leq  v\} \big] .  
	\end{align}
\end{theorem}

Recall that for a random variable $Y\geq 0$ with $\mu=\mathbb{E}[Y]>0$, a random variable $Y^s$ on the same probability space as $Y$ is a size-bias coupling of $Y$ if it satisfies 
\begin{equation} \label{eq:SizeBiasCoupling}
\mathbb{E}[Y f(Y)]=\mu\mathbb{E}[f(Y^s)]
\end{equation}
for all measurable $f$ such that $\mathbb{E}[|Y f(Y)|]<\infty$. Thus, \eqref{Mek} implies that $X+Z+1$ is a size-bias coupling of $X$ so that we can replace $Z$ by $X^s-X-1$ with a size-bias coupling $X^s$ of $X$ in Theorem \ref{Thmmm}. In this form the bound for the total variation distance in \eqref{fsxx} is a classical result (see \cite[Theorem 4.13]{MR2861132} and the discussion at the beginning of Section 5 in \cite{MR3896143} for further references), whose proof  is based on the Chen-Stein method and  size-bias coupling. 
For the Chen-Stein method for Poisson approximation we refer the reader to e.g.\ \cite{MR1163825,MR1632651,MR2861132}, while \cite{MR3896143} is a survey on size-bias coupling.
Note that the bound on the Wasserstein distance in \eqref{fsxx} can be derived by combining the proof of \cite[Theorem 4.13]{MR2861132} with \cite[Theorem 1.1]{MR2274850}.

\begin{remark}\label{remark1}
Let $X$ be as in Theorem \ref{Thmmm} and assume that \eqref{Mek} is satisfied.
\begin{itemize}
\item [(i)] The last expressions on the right-hand sides of \eqref{rrrf} and \eqref{eq:boundOnv} can be further bounded using the inequalities
\begin{align*}
&\mathbb{E}[\vert Z \vert \mathbf{1}\{X- Z_- = k\}] \leq \mathbb{E}[Z_-] + \mathbb{E}[Z_+ \mathbf{1}\{X=k\}], \quad\,\, k\in\N_0, \\
&\mathbb{E}\big[ \vert Z \vert \mathbf{1}\{ X- Z_{-} \leq  v\} \big] \leq \mathbb{E}[Z_-] + \mathbb{E}[Z_+ \mathbf{1}\{X\leq v\}], \quad v\in\N.
\end{align*}
\item [(ii)] 
From \eqref{eq:SizeBiasCoupling} with $f(x)=x$ we obtain $\lambda \mathbb{E}[X^s] =  \mathbb{E}[X^2]$ so that $Z=X^s-X-1$ yields
\begin{align}\label{eq: meanZ}
\mathbb{E}[Z] 
= \frac{1}{\lambda} \big\{ \mathrm{Var}( X) - \lambda\big\}.
\end{align}
\end{itemize}
\end{remark}

The next result constitutes our main achievement and generalizes Theorem \ref{Thmmm}. Instead of assuming that $Z$ satisfies \eqref{Mek} exactly, we allow error terms on the right-hand side of \eqref{Mek}.

\begin{theorem}\label{Thmmm2}
	Let $X$ be a random variable with values in $\mathbb{N}_0$ and let $P_{\lambda}$ be a Poisson random variable with mean $\lambda>0$. Let $Z$ be a random variable defined on the same probability space as $X$ with values in $\Z$, and let   $q_i,i\in\N_0,$ be the sequence given by 
	\begin{align}\label{Mek2}
	  q_{i-1}= i\mathbb{P}(X=i)-\lambda \mathbb{P}(X + Z = i-1), \quad i\in\N.
	\end{align}
	Then,
	\begin{align}\label{fsxx2} 
	& d_{TV}(X,P_\lambda) 
	\leq (1 \wedge \lambda ) \mathbb{E}[ \vert Z \vert ] + \left(1 \wedge \frac{1}{\sqrt{\lambda}}\right) \sum_{i=0}^{\infty} \vert q_i \vert 
	\end{align}
	and
	\begin{align}\label{wass-mainThm}
 d_{W}(X,P_\lambda) 
	\leq  \lambda \mathbb{E}[ \vert Z \vert ] + \sum_{i=0}^{\infty} \vert q_i \vert.
	\end{align}
  Moreover,  if $\mathbb{P}(X + Z \geq 0)=1$, then
	\begin{align}\label{wass-mainThm-2}
		d_{W}(X,P_\lambda) 
		\leq (1.1437 \sqrt{\lambda} \wedge \lambda )\mathbb{E}[ \vert Z \vert ] + \sum_{i=0}^{\infty} \vert q_i \vert,
	\end{align} 
 for all $m\in\N_0$,
	\begin{equation}\begin{split}\label{rrrf2} 
	\vert \mathbb{P}(X=0) - \mathbb{P}(P_\lambda =0) \vert 
	& \leq \frac{m!}{\lambda^{m}}\mathbb{E}\big[ \vert Z \vert \big] + \sum_{k=0}^{m-1} \left( \frac{\lambda}{k+1} \wedge \frac{k!}{\lambda^{k}} \right) \mathbb{E}\big[ \vert Z\vert \mathbf{1}\{X- Z_{ -} = k \} \big] 
	\\
	& \quad + \left(1 \wedge \frac{1}{\lambda}\right) \vert q_0 \vert + \left(1 \wedge \frac{1}{\lambda^2}\right) \sum_{i=1}^\infty \vert q_i \vert 
	\end{split}\end{equation}
	and for all $v\in\N$,
	\begin{equation}\begin{split}\label{bnd-gen-on-v} 
		\vert \mathbb{P}(X\leq v) - \mathbb{P}(P_\lambda \leq v) \vert 
	&\leq \frac{(v+1)^2}{\lambda}\mathbb{E}[ \vert Z \vert ] +  \mathbb{E}\big[ \vert Z \vert \mathbf{1}\{ X- Z_{-} \leq  v\} \big] 
	\\
	&\quad  + \left(1 \wedge \frac{1}{\sqrt{\lambda}}\right) \sum_{i=0}^{\infty} \vert q_i \vert .
	\end{split}\end{equation}
\end{theorem}

Note that Theorem \ref{Thmmm} is a special case of Theorem \ref{Thmmm2}. Indeed, if $q_i=0$ for all $i\in\N_0$, \eqref{Mek2} becomes \eqref{Mek} and the bounds in Theorem \ref{Thmmm2} simplify to those in Theorem \ref{Thmmm}. In this situation $X+Z+1$ is a size-bias coupling of $X$. Thus, we can think of $X+Z+1$ with $Z$ satisfying \eqref{Mek2} as a generalization of size-bias coupling. In order to have good bounds in Theorem \ref{Thmmm2}, the error terms $q_i$, $i\in\N_0$, should be small. The important advantage of Theorem \ref{Thmmm2} compared to Theorem \ref{Thmmm} is that one only needs to construct an \emph{approximate} size-bias coupling instead of an exact size-bias coupling.

For our paper the so-called magic factors or Stein factors play a crucial role. These are bounds on the solutions of the Stein equation, which lead to the factors involving $\lambda$ in our results. Since different classes of test functions have different magic factors, the upper bounds for the differences between $\mathbb P(X\leq v)$ and $\mathbb P(P_\lambda \leq v)$ for $v\in\N_0$ in Theorems \ref{Thmmm} and  \ref{Thmmm2} are of a better order in $\lambda$ than those for the total variation distance or the Wasserstein distance. This observation is essential for obtaining approximation results in the Kolmogorov distance as \eqref{eq:InterpointDistances} since it allows to bound the right-hand sides of \eqref{rrrf}, \eqref{eq:boundOnv}, \eqref{rrrf2} and \eqref{bnd-gen-on-v} uniformly in $\lambda$. For a different Poisson approximation result where one has a better order in $\lambda$ for the difference of the probabilities at zero than for the total variation distance we refer the reader to \cite[Theorem 1]{MR972770}.

To demonstrate the versatility of our general main results we apply them to several examples. In particular, we deduce bounds as \eqref{eq:InterpointDistances}, where we compare minima or maxima of collections of dependent random variables with random variables having an exponential, Weibull or Gumbel distribution.   

We study the Poisson approximation of $U$-statistics constructed from an underlying Poisson or binomial point process (see Subsections \ref{U-stt-B-prf} and \ref{U-stt-P-prf}). 
As application of our main finding on $U$-statistics  with Poisson input, Theorem \ref{U-stat-fin}, we consider the minimum inter-point distance problem discussed at the beginning of the introduction and establish the bound \eqref{eq:InterpointDistances} for the exponential approximation in Kolmogorov distance. 

Our next example is the Poisson approximation of the number of non-overlapping $k$-runs in a sequence of $n$ i.i.d.\ Bernoulli random variables. By a $k$-run one means at least $k$ successes in a row. Here, we use Theorem \ref{Thmmm} to bound the difference between the probability that among $n$ trials there are no more than $v$ non-overlapping $k$-runs and $\mathbb P (P_\alpha\leq v)$ for a certain Poisson random variable $P_\alpha$; this bound is remarkable because it does not depend on $k$, i.e., the number of required successes in a row.

For stationary Poisson-Voronoi tessellations we consider statistics related to circumscribed radii and inradii. We use the inequality \eqref{rrrf2} in Theorem \ref{Thmmm2} to compare a transform of the minimal circumscribed radius of the cells with the nucleus in an observation window with a Weibull random variable in Kolmogorov distance. For this example we use the full generality of Theorem \ref{Thmmm2} since we construct a coupling that satisfies \eqref{Mek2}, but is not a size-bias coupling. By applying the inequality \eqref{rrrf} in  Theorem \ref{Thmmm} we approximate a transform of the maximal inradius of the cells with the nucleus in an observation window by a Gumbel random variable in the Kolmogorov distance.

A crucial contribution of this paper to stochastic geometry is that we provide bounds with respect to the Kolmogorov distance for the distributional approximation of some minima and maxima. The limiting distributions of the minimal distance between the points of a Poisson process and of large inradii and small circumscribed radii of Poisson-Voronoi tessellations have been studied before in e.g.\ \cite{MR3252817,CHENAVIER20142917,MR2971726,MR3585403}. Some of these works provide quantitative bounds for the difference of the distribution functions at a fixed $u\in\R$, which depend on $u$. Thanks to our general Poisson approximation results Theorem \ref{Thmmm} and Theorem \ref{Thmmm2}, we are able to derive uniform bounds for all $u\in\R$. An alternative approach to deducing such results via Poisson approximation is to apply directly Stein's method for the exponential, Weibull or Gumbel distribution; see e.g.\ \cite{MR2861132} for more details on Stein's method for exponential approximation. 

In \cite{MR3780386}, a general result for the Poisson approximation of statistics of Poisson processes is derived by combining the Chen-Stein method and a kind of size-bias coupling and applied to study some statistics of inhomogeneous random graphs such as isolated vertices. Requiring some (stochastic) ordering assumptions between a random variable and its size-bias coupling leads to Poisson approximation results. In a similar spirit to our work,   these ordering conditions were relaxed in \cite{MR3877881}. For some recent Poisson process convergence results related to stochastic geometry we refer the reader to \cite{otto2020poisson,pianoforte2021criteria}.

Other noteworthy general results derived in this paper are  lower and upper bounds on the probability that $X$ equals $0$, which are given in Proposition \ref{prop-bound-0} and Corollary \ref{prpbound1}. Informally,  they say that $\mathbb{P}(X=0)$ can be bounded from above or below by $e^{-\lambda}$ for some $\lambda>0$ if the random variable $Z$ and the sequence $q_i, i\in\N_0,$ in Theorem  \ref{Thmmm} and Theorem \ref{Thmmm2}  satisfy certain conditions on their signs;   for  $Z$ as in Theorem \ref{Thmmm},  it is understood that $q_i= 0$ for all $i\in\N_0$. These results sometimes allow us to remove the absolute values from the left-hand sides of \eqref{rrrf} and \eqref{rrrf2}.

The proof of Theorem \ref{Thmmm2} is based on the Chen-Stein method and the coupling in \eqref{Mek2}. Using the solution of the Stein equation for the Poisson distribution, we derive in Proposition \ref{oid} a new expression for
the difference $	\vert \mathbb{E}[g(P_\lambda)] - \mathbb{E}[g(X)] \vert$
for any $g\in\text{Lip}(1)$. Taking in Proposition \ref{oid}, the supremum over all functions in $\mathrm{Lip}(1)$ (or all indicator functions) establishes a different way to represent the Wasserstein distance (or the total variation distance). Moreover, choosing $g=\mathbf 1\{\cdot \leq v\}$ with $v\in\N$ gives a new expression for $\vert\mathbb{P}(X\leq v)-\mathbb{P}(P_{\lambda}\leq v)\vert$. These  identities are then manipulated and combined with the magic factors and the coupling in \eqref{Mek2} to prove Theorem \ref{Thmmm2}.

Before we present our applications in Section \ref{Sec: applic}, we prove our main results in the next
section.

\section{Proof of the main results}
This section provides the proofs of Theorem \ref{Thmmm} and Theorem \ref{Thmmm2}. To this end, we first study  the Stein equation for Poisson random variables. For any fixed $g\in \mathrm{Lip}(1)$, the solution of the Stein equation  is a function $f_g:\N_0\rightarrow \R$ with $f_g(0)=0$ that satisfies 
\begin{align}\label{steinEq-W}
	\lambda f_g(i+1) - i f_g(i) 
	= g(i) - \mathbb{E}[g(P_\lambda)], \quad i\in\mathbb{N}_0.
\end{align}
The function $f_g$ can be obtained by solving \eqref{steinEq-W} recursively for $i=0,1,\dots$. An explicit expression for  this solution is given in  \cite[Lemma 1]{MR974580}. In particular, for $g=\mathbf{1}_A$ with $A\subset\N_0$, one has the following representation for $f_g$ (see e.g.\ \cite[Lemma 4.2]{MR2861132}). 
\begin{lemma}\label{LemmaSteinEq}
For any  $\lambda>0$ and $A\subset\mathbb{N}_0$ the unique solution $f_A$ of
\begin{align}\label{steinEq}
\lambda f_A(i+1) - i f_A(i) 
= \mathbf{1}\{i\in A\} - \mathbb{P}(P_\lambda\in A), \quad i\in\mathbb{N}_0,
\end{align}
with  $f_A(0)=0$ is given by 
$$
f_A(i) 
=  \frac{e^{\lambda}(i-1)!}{\lambda^{i}} \big[\mathbb{P}(P_\lambda \in A\cap \{0,1,\dots, i-1 \} ) - \mathbb{P}(P_\lambda\in A) \mathbb{P}(P_\lambda \leq i-1)\big], 
\quad i\in\N.
$$
\end{lemma}
From now on, we denote by $f_A$ the solution of the Stein equation  \eqref{steinEq-W} for $g=\mathbf{1}_A$ with $A\subset \N_0$.
Let $X$ be a random variable with values in $\N_0$. The idea of the Chen-Stein method for the Poisson approximation  of $X$ is to plug $X$ in \eqref{steinEq-W} and to take the expectation, which yields
$$
\mathbb{E}[\lambda f_g(X+1) - X f_g(X) ] = \mathbb{E}[g(X)] - \mathbb{E}[g(P_\lambda )].
$$
So we can control the difference between the  expectations of $g(X)$ and $g(P_{\lambda})$ on the right-hand side by estimating the term on the left-hand side. This requires some bounds on the solution of \eqref{steinEq-W}, which we give in the sequel.  For a function $h:\N_0\to\R$ we define $\Delta h: \N_0\to\R$ by $\Delta h(i)=h(i+1)-h(i)$. The solution of the Stein equation \eqref{steinEq-W} and its differences can be bounded by the following terms, which are called magic factors or Stein factors (see \cite[Theorem 1.1]{MR2274850}).

\begin{lemma}\label{thm: bndfg} 
	Let $f_g$ be the solution of \eqref{steinEq-W}. Then,
	$$
	\max_{i\in\mathbb{N}_0} \vert f_g(i) \vert
	\leq 1  
	\quad
	\text{and}
	\quad
	\max_{i\in\mathbb{N}} |\Delta f_g(i)| 
	\leq  1\wedge \frac{8}{3\sqrt{2e\lambda}} \leq 1\wedge \frac{1.1437}{\sqrt{\lambda}} .
	$$
\end{lemma}
Since $f_g(0)=0$, Lemma \ref{thm: bndfg} implies that 
\begin{align}\label{eq:  fNo}
	\max_{i\in\mathbb{N}_0} \vert f_g(i) \vert
	\leq 1   \quad \text{ and } \quad \max_{i\in\mathbb{N}_0} |\Delta f_g(i)|  \leq 1.
\end{align}
Moreover, the solution of \eqref{steinEq} for $A\subset\N_0$ has the following magic factors (see e.g.\  \cite[Lemma 4.4]{MR2861132}).
\begin{lemma}\label{lem:BoundfA}
For $f_A$ as in Lemma \ref{LemmaSteinEq},
$$
\max_{i\in\mathbb{N}_0} \vert f_{A}(i) \vert
\leq 1 \wedge \frac{1}{\sqrt{\lambda}}
\quad
\text{and}
\quad
\max_{i\in\mathbb{N}_0} |\Delta f_A(i)| 
\leq 1\wedge \frac{1}{\lambda}.
$$
\end{lemma}

We now derive similar - potentially sharper - magic factors for the special cases  $A=\{0,\dots , v\}, v\in\N_0$. Similar bounds for sets $A$ that are singletons were deduced  for the translated Poisson approximation  in \cite[Lemma 3.7]{MR2358635}.

\begin{lemma}\label{lemma}
Let $f_{\{0\}}$ be the unique solution of \eqref{steinEq} for $A=\{0\}$. Then,  
\begin{align}\label{boundf_0}
\vert  f_{\{0\}}(i) \vert 
\leq 
\begin{cases}
  1 \wedge \frac{1}{\lambda}, & \quad \mathrm{if} \,\, i=1 , \\
	1 \wedge \frac{1}{\lambda^2}, & \quad \mathrm{if} \,\, i\geq 2 ,
\end{cases}
\end{align}
and for all $i\in\mathbb{N}$,
\begin{align}\label{Deltaf0Negative}
\Delta f_{\{0\}} (i) 
\leq 0 .
\end{align} 
Furthermore for all $i,n\in\mathbb{N}$ with $i\geq n$,
\begin{align}\label{eqn:BoundDelta}
\vert \Delta f_{\{0\}}(i) \vert 
\leq   \frac{1}{n} \wedge \frac{(n-1)!}{\lambda^n} .
\end{align}
Let $f_{\{0,\dots,v\}}$  be the unique solution of \eqref{steinEq} for $A=\{0,\dots,v\}$ with $v\in\N$ and $v\leq \lambda$. Then for all $i\geq v+2$,
\begin{align}\label{eq: bndOndeltav}
\Delta f_{\{0,\dots, v\}} (i) \leq 1 \wedge \frac{(v + 1)^2}{\lambda^2}.
\end{align}
\end{lemma}
\begin{proof}
Obviously, the upper bound $1$ in \eqref{boundf_0} follows from Lemma \ref{lem:BoundfA}. Lemma \ref{LemmaSteinEq} yields for $i\in\mathbb{N}$ that
\begin{equation}\label{eqn:f0}
f_{\{0\}}(i)
= \frac{(i-1)!}{\lambda^{i}} (1-\mathbb{P}(P_\lambda\leq i-1))
= \frac{(i-1)!}{\lambda^{i}} \sum_{m=i}^\infty \frac{\lambda^m}{m!}e^{-\lambda}
= \sum_{\ell=0}^\infty \frac{\lambda^\ell}{(i+\ell)!}(i-1)!e^{-\lambda}.
\end{equation}
This implies \eqref{boundf_0} for  $i=1,2$, and yields for $i\geq 3$ that
\begin{align*}
f_{\{0\}}(i)
= \sum_{\ell=0}^\infty \frac{\lambda^\ell}{(i+\ell)!}(i-1)!e^{-\lambda} = \frac{1}{\lambda^2}\sum_{\ell=0}^\infty \frac{\lambda^{\ell+2}}{(\ell+2)!}\frac{(i-1)!(\ell+2)!}{(i+\ell)!}e^{-\lambda}.
\end{align*}
Thus, the elementary inequalities
$$
\frac{(i-1)!(\ell+2)!}{(i+\ell)!} = \frac{(i-1)!}{(\ell+3) \cdot \hdots \cdot (\ell+i)}\leq \frac{2 (i-1)!}{i!}\leq 1  
$$
establish \eqref{boundf_0} for $i\geq3$.
From \eqref{eqn:f0} we also obtain for $n\in\N$,
\begin{align*}
\Delta f_{\{0\}}(i) & = \sum_{\ell=0}^\infty \bigg( \frac{\lambda^\ell}{(i+1+\ell)!}i! - \frac{\lambda^\ell}{(i+\ell)!}(i-1)! \bigg) e^{-\lambda} 
\\
& = \sum_{\ell=0}^\infty  \frac{\lambda^\ell}{(i+1+\ell)!} \big( i! - (i+1+\ell)(i-1)!\big) e^{-\lambda} 
\\
& = - \sum_{\ell=0}^\infty  \frac{\lambda^\ell}{(i+1+\ell)!} (\ell+1)(i-1)! e^{-\lambda} 
\\
& = - \sum_{\ell=0}^\infty  \frac{\lambda^\ell}{(n+\ell)!}\frac{(\ell+1) (n+\ell)! (i-1)!}{(i+1+\ell)!}  e^{-\lambda},
\end{align*}
which proves \eqref{Deltaf0Negative}. For $i,n\in\N$ with $i\geq n$ the elementary inequalities
$$
\frac{(\ell+1) (n+\ell)! (i-1)!}{(i+1+\ell)!} 
\leq \frac{(n+\ell)! (i-1)!}{(i+\ell)!} 
\leq (n-1)!
$$
lead to
$$
\big\vert \Delta f_{\{0\}}(i) \big\vert 
\leq (n-1)!  e^{-\lambda} \sum_{\ell=0}^\infty  \frac{\lambda^\ell}{(n+\ell)!} .
$$
Now the observations that
$$
\sum_{\ell=0}^\infty  \frac{\lambda^\ell}{(n+\ell)!}  \leq \frac{e^{\lambda}}{\lambda^n} \quad \text{and} \quad
\sum_{\ell=0}^\infty  \frac{\lambda^\ell}{(n+\ell)!} \leq \frac{1}{n!} \sum_{\ell=0}^\infty  \frac{\lambda^\ell}{\ell!} \frac{\ell! n!}{(n+\ell)!} \leq \frac{1}{n!} \sum_{\ell=0}^\infty  \frac{\lambda^\ell}{\ell!}  \leq \frac{e^{\lambda}}{n!}
$$
show \eqref{eqn:BoundDelta}. Finally assume $\lambda\geq v$. By Lemma \ref{LemmaSteinEq},   we obtain  for  $i\geq v+2$,
$$
\Delta f_{\{0,\dots, v\}} (i) = e^\lambda \mathbb{P}(P_\lambda \in \{ 0,\dots , v\}) \Delta f_{\{0\}}(i).
$$
Then \eqref{eqn:BoundDelta} with $n=v+2$ implies that  
$$
\vert \Delta f_{\{0,\dots, v\}} (i)\vert 
\leq \frac{(v+1)!}{\lambda^{v+2}}\sum_{\ell =0}^v \frac{\lambda^\ell}{\ell!}
= \frac{(v+1)!}{\lambda^{2}}\sum_{\ell =0}^v \frac{\lambda^{\ell-v}}{\ell!}\leq  \frac{(v+1)^2}{\lambda^2}, 
$$
where we used the inequality $ \lambda^{\ell-v}/\ell!\leq 1/ v!$ for $\ell=0,\dots, v$ and    $\lambda\geq v$ in the last step. This and Lemma \ref{lem:BoundfA} establish \eqref{eq: bndOndeltav}.
\end{proof}

The next proposition compares the distributions of an integer-valued random variable and a Poisson distributed random variable.

\begin{proposition}\label{oid}
Let $X$ be a random variable taking values in $\mathbb{N}_0$, let $\lambda\in(0,\infty)$, and define
$$
\mathcal{D}(i)
=i\mathbb{P}(X=i)-\lambda\mathbb{P}(X=i-1), 
\quad i\in\mathbb{N}.
$$
Then, for all   $g\in\mathrm{Lip}(1)$,
$$
\mathbb{E}[g(P_\lambda)] - \mathbb{E}[g(X)]
= \sum_{i=1}^\infty  f_{g}(i) \mathcal{D}(i),
$$
where   $f_g$ is the solution of \eqref{steinEq-W}.
\end{proposition}

\begin{proof}
It follows from Lemma \ref{LemmaSteinEq} and the definition of $\mathcal{D}(i)$, $i\in\mathbb{N}$, that
\begin{align*}
\mathbb{E}[g(P_\lambda)] - \mathbb{E}[g(X)]
& = \mathbb{E}[ X f_g(X)- \lambda f_g(X+1)]  
=  \sum_{i=0}^\infty \mathbb{P}(X=i)(if_g(i) -\lambda f_g (i+1)) 
\\
& =  \sum_{i=1}^\infty \mathbb{P}(X=i) i f_g(i) - \sum_{i=1}^\infty \mathbb{P}(X=i-1)\lambda f_g(i) 
= \sum_{i=1}^\infty  f_{g}(i) \mathcal{D}(i),
\end{align*}
which gives the desired result.  
\end{proof}
From Proposition \ref{oid} we derive the identity
$$
\mathbb{P}(P_\lambda\in A) - \mathbb{P}(X\in A) 
= \sum_{i=1}^\infty  f_{A}(i) \mathcal{D}(i), \quad A\subset \N_0,
$$
where $f_A$  is the solution of \eqref{steinEq}.
We are now in position to show Theorem \ref{Thmmm2}.

\begin{proof}[Proof of Theorem \ref{Thmmm2}]
It follows from \eqref{Mek2} that
$$
\mathcal{D}(i)
=i\mathbb{P}(X=i)-\lambda\mathbb{P}(X=i-1) 
= \lambda \mathbb{P}(X + Z= i-1) -\lambda\mathbb{P}(X=i-1) + q_{i-1}  , 
\quad i\in\mathbb{N}.
$$
Thus, Proposition \ref{oid} yields for $g\in\mathrm{Lip}(1)$ that 
\begin{equation}\label{eqn:DecompositionHQ-W}
\begin{split}
& \mathbb{E}[g(P_\lambda)] - \mathbb{E}[g(X)]\\
& =  \lambda \sum_{i=1}^\infty  f_{g}(i) \big( \mathbb{P}(X + Z= i-1) - \mathbb{P}(X = i-1)\big)  +  \sum_{i=1}^\infty f_{g}(i) q_{i-1}
 =: H_g + Q_g .
\end{split}
\end{equation}
With $f_g(0)=0$ and the convention $f_g(i)=0$ for $i<0$, we obtain
\begin{equation}\label{eqn:HA}
\begin{split}
H_g & = \lambda \sum_{j\in\Z\setminus\{0\}}\sum_{i=1}^\infty f_{g}(i) \big(   \mathbb{P}(X = i-1 -j, Z =j) - \mathbb{P}(X = i-1, Z =j) \big) 
\\
& = \lambda \sum_{j\in\Z\setminus\{0\}}\sum_{i\in\Z} f_{g}(i) \big(   \mathbb{P}(X = i-1 -j, Z =j) - \mathbb{P}(X = i-1, Z = j) \big)
\allowdisplaybreaks
\\
& = \lambda \sum_{j\in\Z\setminus\{0\}}\sum_{i\in\Z} f_{g}(i+j) \mathbb{P}(X = i-1 , Z =j) - f_{g}(i) \mathbb{P}(X = i-1, Z =j)
\allowdisplaybreaks
\\
& = \lambda \sum_{j\in\Z\setminus\{0\}}\sum_{i\in\Z} \big( f_{g}(i+j) - f_{g}(i) \big) \mathbb{P}(X = i-1 , Z =j)
\\
& = \lambda \sum_{j\in\Z\setminus\{0\}}\sum_{i=1}^\infty \big( f_{g}(i+j) - f_{g}(i) \big) \mathbb{P}(X = i-1 , Z =j),
\end{split}
\end{equation}
where we used that $X$ takes only values in $\N_0$ in the last step. The triangle inequality implies that 
\begin{equation*}
\begin{split}
|H_g| & \leq     \lambda  \max_{i\in\mathbb{N}_0} |\Delta f_g(i)|  \sum_{j\in\Z\setminus\{0\}}\sum_{i=1}^\infty \vert j \vert \mathbb{P}(X = i-1 , Z =j)
\\
& =   \lambda  \max_{i\in\mathbb{N}_0} |\Delta f_g(i)|   \sum_{j\in\Z\setminus\{0\}} \vert j \vert \mathbb{P}( Z =j) 
=  \lambda  \max_{i\in\mathbb{N}_0} |\Delta f_g(i)|   \mathbb{E}[\vert Z \vert ]. 
\end{split}
\end{equation*}
Furthermore, we have
\begin{equation}\label{bndQ_g}
|Q_g| \leq \max_{i\in\mathbb{N}} | f_g(i)|\sum_{i=0}^\infty \vert q_i \vert.
\end{equation}
Then  combining  \eqref{eq:  fNo} and the bounds on $\vert H_g \vert $ and $\vert Q_g\vert$ establishes \eqref{wass-mainThm}. Moreover, from  Lemma \ref{lem:BoundfA} and the bounds on $\vert H_g \vert $ and $\vert Q_g\vert$ with $g=\mathbf{1}_A$ for $A\subset\N_0$, we obtain \eqref{fsxx2}. 

Next, we notice that 
\begin{equation}\label{eqn: decompH_g}
	\begin{split}
		|H_g| 
		& \leq  \lambda \sum_{j=1}^\infty\sum_{i=1}^\infty \vert f_{g}(i+j) - f_{g}(i) \vert\mathbb{P}(X = i-1 , Z =j) 
		\\
		& \quad +  \lambda \sum_{j=1}^\infty\sum_{i=1}^\infty \vert f_{g}(i-j) - f_{g}(i) \vert \mathbb{P}(X = i-1 , Z =-j).
	\end{split}
\end{equation}
The assumption $\mathbb{P}(X+Z\geq 0)=1$ implies that $\mathbb{P}(X = i-1, Z=-j)=0$ for all $i,j\in\mathbb{N}$ with $i\leq j$.
Hence, we obtain
\begin{equation}\begin{split}\label{eqn: repr H_g2}
&\lambda \sum_{j=1}^\infty\sum_{i=1}^\infty \vert f_{g}(i-j) - f_{g}(i) \vert \mathbb{P}(X = i-1 , Z =-j)
\\
&= \lambda \sum_{j=1}^\infty\sum_{i=j+1}^\infty \big\vert f_{g}(i-j)- f_{g}(i)\big\vert \mathbb{P}(X = i-1, Z=-j) 
	\\ 
	& = \lambda \sum_{j=1}^\infty\sum_{i=1}^\infty \big\vert f_{g}(i)- f_{g}(i+j)\big\vert \mathbb{P}(X = i+ j-1, Z=-j). 
\end{split}\end{equation}
From \eqref{eqn: decompH_g}, \eqref{eqn: repr H_g2} and the triangle inequality it follows that
\begin{align*}
		|H_g| 
& \leq \lambda  \max_{i\in\mathbb{N}} |\Delta f_g(i)|  \sum_{j=1}^\infty j\Big(\sum_{i=1}^\infty  \mathbb{P}(X = i-1 , Z =j) + \sum_{i=1}^\infty  \mathbb{P}(X = i+j-1 , Z =-j)\Big)
\\
& \leq   \lambda  \max_{i\in\mathbb{N}} |\Delta f_g(i)|   \sum_{j\in\Z\setminus\{0\}} \vert j \vert \mathbb{P}( Z =j) 
=  \lambda  \max_{i\in\mathbb{N}} |\Delta f_g(i)|   \mathbb{E}[\vert Z \vert ].
\end{align*}
Together with \eqref{eqn:DecompositionHQ-W} and \eqref{bndQ_g}, this implies that
$$
\vert \mathbb{E}[g(P_\lambda)] - \mathbb{E}[g(X)] \vert \leq \lambda  \max_{i\in\mathbb{N}} |\Delta f_g(i)|   \mathbb{E}[\vert Z \vert ] + \max_{i\in\mathbb{N}_0} | f_g(i)|\sum_{i=0}^\infty \vert q_i \vert.
$$
Hence, Lemma \ref{thm: bndfg}   establishes \eqref{wass-mainThm-2}.

Combining \eqref{eqn:DecompositionHQ-W}, \eqref{eqn: decompH_g} and \eqref{eqn: repr H_g2} with $g=\mathbf{1}_A$ for $A\subset\N_0$ yields
\begin{align}\label{eqn:DecompositionHQ}
	\mathbb{P}(P_\lambda\in A) - \mathbb{P}(X\in A) =: H_A + Q_A
\end{align}
where $H_A= H_{g}$ and $Q_A= Q_{g}$ with $g=\mathbf{1}_A$, and
\begin{equation*}
	\begin{split}
		|H_A| 
		& \leq  \lambda \sum_{j=1}^\infty\sum_{i=1}^\infty \vert f_{A}(i+j) - f_{A}(i) \vert\mathbb{P}(X = i-1 , Z =j) 
		\\
		& \quad +  \lambda \sum_{j=1}^\infty\sum_{i=1}^\infty \vert f_{A}(i) - f_{A}(i+j) \vert \mathbb{P}(X = i+j-1 , Z =-j)
		=: H^{(1)}_A + H^{(2)}_A.
	\end{split}
\end{equation*}
For $A=\{0\}$, by \eqref{eqn:BoundDelta} in Lemma \ref{lemma} with $n=i$ for $i\leq m$ and $n=m+1$ for $i\geq m+1$, we have
\begin{equation*} 
\begin{split} 
H^{(1)}_{\{0\}}
& \leq  \sum_{j=1}^\infty  \sum_{i=1}^{m}  \bigg( \frac{\lambda}{i} \wedge \frac{(i-1)!}{\lambda^{i-1}} \bigg)  j \mathbb{P}(X = i-1, Z =j) 
+ \sum_{j=1}^\infty  \sum_{i=m+1}^{\infty} \frac{m!}{\lambda^{m}} j \mathbb{P}(X = i-1, Z =j) 
\\
& = \sum_{k=0}^{m-1}\bigg( \frac{\lambda}{k+1} \wedge \frac{k!}{\lambda^{k}} \bigg) \mathbb{E}[ Z_+  \mathbf{1}\{X=k\}] + \frac{m!}{\lambda^{m}}\mathbb{E}[ Z_+ \mathbf{1}\{X\geq m\}]. 
\end{split}
\end{equation*}
Again \eqref{eqn:BoundDelta} in Lemma \ref{lemma} with $n=i$ for $i\leq m$ and $n=m+1$  for $i\geq m+1$ leads to
\begin{equation*}
\begin{split}
H^{(2)}_{\{0\}}
& \leq  \sum_{j=1}^\infty  \sum_{i=1}^{m}  \bigg( \frac{\lambda}{i} \wedge \frac{(i-1)!}{\lambda^{i-1}} \bigg)    j  \mathbb{P}(X = i+ j -1, Z=-j) 
\\
& \quad + \sum_{j=1}^\infty  \sum_{i=m+1}^{\infty} \frac{m!}{\lambda^{m}}   j  \mathbb{P}(X = i+j -1, Z=-j) 
\\
& = \sum_{i=1}^{m}  \bigg( \frac{\lambda}{i} \wedge \frac{(i-1)!}{\lambda^{i-1}} \bigg) \mathbb{E}[ Z_- \mathbf{1}\{X + Z= i-1\}] + \frac{m!}{\lambda^{m}} \mathbb{E}[ Z_- \mathbf{1}\{X + Z \geq m\}] 
\\
& = \sum_{k=0}^{m-1}  \bigg( \frac{\lambda}{k+1} \wedge \frac{k!}{\lambda^{k}} \bigg) \mathbb{E}[ Z_- \mathbf{1}\{X + Z= k\}] + \frac{m!}{\lambda^{m}} \mathbb{E}[ Z_- \mathbf{1}\{X + Z \geq m\}].
\end{split}
\end{equation*}
From \eqref{boundf_0} in Lemma \ref{lemma} it follows that
\begin{equation*}
|Q_{\{0\}}|
\leq \bigg(1 \wedge \frac{1}{\lambda}\bigg) \vert q_0 \vert  + \bigg(1 \wedge \frac{1}{\lambda^2}\bigg) \sum_{i=1}^{\infty} \vert q_i \vert .
\end{equation*}
Combining \eqref{eqn:DecompositionHQ} and the  bounds on $|Q_{\{0\}}|, H^{(1)}_{\{0\}} $ and  $H^{(2)}_{\{0\}}$ completes the proof of \eqref{rrrf2}. 

For $\lambda<v$, \eqref{bnd-gen-on-v} follows directly from \eqref{fsxx2}. 
By Lemma \ref{lem:BoundfA} for $i\leq v+1$ and \eqref{eq: bndOndeltav} in Lemma \ref{lemma} for $i\geq v+2$, we obtain 
\begin{align*} 
H^{(1)}_{\{0,\dots, v\}} 
& \leq   (1 \wedge \lambda) \sum_{j=1}^\infty\sum_{i=1}^{v+1} j \mathbb{P}(X = i-1 , Z =j) 
  +  \sum_{j=1}^\infty\sum_{i=v+2}^{\infty} \frac{ (v+1)^2 }{\lambda} j \mathbb{P}(X = i-1 , Z =j) 
\notag
\\
& =  (1 \wedge \lambda) \mathbb{E}[Z_+ \mathbf{1}\{X\leq v\}] +  \frac{(v+1)^2}{\lambda}\mathbb{E}[Z_+ \mathbf{1}\{ X\geq	 v+1\}]
\end{align*}
and
\begin{align*}
H^{(2)}_{\{0,\dots, v\}}
& \leq   (1\wedge \lambda)\sum_{j=1}^\infty  \sum_{i=1}^{v+1}      j  \mathbb{P}(X = i+ j -1, Z=-j) \\
&\quad +
  \sum_{j=1}^\infty  \sum_{i=v+2}^{\infty}  \frac{ (v+1)^2 }{\lambda }  j  \mathbb{P}(X = i+j -1, Z=-j) 
\\
& = (1\wedge \lambda)  \mathbb{E}[ Z_- \mathbf{1}\{X + Z\leq v\}] +  \frac{(v+1)^2}{\lambda } \mathbb{E}[ Z_- \mathbf{1}\{X + Z \geq v+1\}] .
 \end{align*}
Moreover, Lemma \ref{lem:BoundfA} yields
$$
\vert Q_{\{0,\dots , v\}} \vert \leq \max_{i\in\mathbb{N}_0} | f_{\{0,\dots, v\}}(i)|\sum_{i=0}^\infty \vert q_i \vert \leq \left(1 \wedge \frac{1}{\sqrt{\lambda}}\right) \sum_{i=0}^{\infty} \vert q_i \vert.
$$
Combining \eqref{eqn:DecompositionHQ} with $A=\{0,\dots, v\}$ and the bounds on $\vert Q_{\{0,\dots , v\}} \vert,  H^{(1)}_{\{0,\dots, v\}} $ and $H^{(2)}_{\{0,\dots, v\}}$ establishes \eqref{bnd-gen-on-v}.
\end{proof}

Next we derive Theorem \ref{Thmmm} from Theorem \ref{Thmmm2}.

\begin{proof}[Proof of Theorem \ref{Thmmm}]
It follows from \eqref{Mek} that $X$ and $Z$ satisfy \eqref{Mek2} with $\lambda=\mathbb{E}[X]$ and $q_i=0$ for $i\in\N_0$ and that
\begin{align*}
\lambda 
= \mathbb{E}[X]
= \sum_{k=1}^\infty k \mathbb{P}(X=k) 
= \sum_{k=1}^\infty \lambda \mathbb{P}(X+Z=k-1)
=\lambda \mathbb{P}(X + Z \geq 0),
\end{align*}
whence $\mathbb{P}(X + Z\geq 0)=1$. This allows us to apply Theorem \ref{Thmmm2} which proves \eqref{fsxx}, \eqref{rrrf} and \eqref{eq:boundOnv}.
\end{proof}

The next result provides some inequalities for the probability that a non-negative integer-valued random variable equals zero.

\begin{proposition}\label{prop-bound-0}
Let $X$ be a random variable with values in $\mathbb{N}_0$ and $\lambda>0$. Consider a random variable $Z$ defined on the same probability space as $X$ with values in $\Z$, and let  $(q_i)_{i\in\N_0}$ be the sequence given by
	\begin{align*}
	q_{i-1}= i\mathbb{P}(X=i)-\lambda \mathbb{P}(X + Z = i-1), \quad i\in\N.
\end{align*}
\begin{itemize}
\item [a)] If $Z$ is non-negative and $q_i\leq 0$ for $i\in\N_0$,
\begin{align*}
\mathbb{P}(X=0) 
\geq e^{-\lambda}.
\end{align*}
\item [b)] If $Z$ is non-positive, $\mathbb{P}(X + Z \geq 0)=1$ and $q_i \geq 0$ for $i\in\N_0$, 
\begin{align*}
\mathbb{P}(X=0) 
\leq e^{-\lambda}.
\end{align*}
\end{itemize}
\end{proposition}

\begin{proof}
It follows from \eqref{eqn:DecompositionHQ-W} and \eqref{eqn:HA} for $f=\mathbf{1}_{\{0\}}$ as well as $\mathbb{P}(P_\lambda=0)=e^{-\lambda}$ that
\begin{align*}
& e^{-\lambda}- \mathbb{P}(X=0) \\ 
& = \lambda \sum_{j\in\mathbb{Z}\setminus\{0\}}  \sum_{i=1}^\infty \big(f_{\{0\}}(i+j) - f_{\{0\}}(i)\big)\mathbb{P}(X = i-1, Z=j) + \sum_{i=1}^\infty f_{\{0\}}(i)q_{i-1} .
\end{align*}
By the assumption that $Z\geq 0$ (resp.\ $Z\leq 0$ and $\mathbb{P}(X + Z \geq 0)=1$) the first sum on the right-hand side runs only over $j\geq 1$ (resp.\ $j\leq -1$ and the inner sum runs over all $i\in\N$ with $i+j\geq 1$). Together with
$$
f_{\{0\}}(i+j) - f_{\{0\}}(i) \leq 0 \ \text{for} \ i,j\geq 1 \quad \text{and} \quad  f_{\{0\}}(i+j) - f_{\{0\}}(i) \geq 0 \ \text{for} \ j\leq -1,\,\, i+j\geq 1,
$$
which follows from \eqref{Deltaf0Negative} in Lemma \ref{lemma}, and the assumptions on $(q_i)_{i\in\N_0}$,  this leads to the desired results. 
\end{proof}

Since \eqref{Mek} is a special case of \eqref{Mek2} with $\mathbb{P}(X + Z\geq 0)=1$ (see the proof of Theorem \ref{Thmmm}), the following corollary is a direct consequence of Proposition \ref{prop-bound-0}.

\begin{corollary}\label{prpbound1}
Let $X$ be a random variable taking values in $\mathbb{N}_0$ and  let $\lambda=\mathbb{E}[X]>0$. Assume there exists a  random variable $Z$ such that \eqref{Mek} is satisfied. 
\begin{itemize}
\item [a)] If $Z$ is non-negative,
\begin{align*}
\mathbb{P}(X=0) 
\geq e^{-\lambda}.
\end{align*}
\item [b)] If $Z$ is non-positive,
\begin{align*}
\mathbb{P}(X=0) 
\leq e^{-\lambda}.
\end{align*}
\end{itemize}
\end{corollary}

\section{Applications}\label{Sec: applic}

\subsection{$U$-statistics of binomial point processes}\label{U-stt-B-prf}

Let $(\mathbb X, \mathcal X)$ be a measurable space. A point process on $\mathbb X$ is a random element in the set of all $\sigma$-finite counting measures on $\mathbb X$, denoted by $\mathbf N_{\mathbb X}$, which is  measurable with respect to the $\sigma$-field generated by the sets of the form
\begin{align*}
	\{
	\mu\in  \mathbf N_{\mathbb X} \, : \, \mu(B)=k
	\} , \quad k\in\N_0, B\in \mathcal{X}.
\end{align*}
We consider a  binomial point process $\beta_{n}$ on $\mathbb X$ of $n\in\N$ independent points in $\mathbb X$ that are distributed according to a probability measure $K$. Let $\ell\in\N$ and let $h: \mathbb{X}^\ell\to\{0,1\}$ be a measurable symmetric function. In the following we study the $U$-statistic
$$
S=\frac{1}{\ell!} \sum_{(x_{1},...,x_{\ell})\in\beta_{n,\neq}^{\ell}} h(x_{1},...,x_{\ell}),
$$
where $\beta_{n,\neq}^{\ell}$ denotes the set of all $\ell$-tuples of distinct points of $\beta_n$.  We refer to the monographs \cite{MR1472486,MR1075417} for more details on $U$-statistics and their applications in statistics. A straightforward computation shows that
$$
\lambda:=\mathbb{E}[S]=\frac{(n)_\ell}{\ell!}\int_{\mathbb{X}^\ell} h(x_{1},\dots,x_{\ell}) d K^\ell(x_{1},\dots,x_{\ell}),
$$
where $(n)_\ell $ stands for the $\ell$-th descending factorial.

In this subsection, we establish bounds on the Poisson approximation of $S$ in the total variation and Wasserstein distances. We also provide  bounds on the pointwise difference between the cumulative distribution functions of $S$ and $P_{\lambda}$. To this end, we define
$$
 r	= \underset{1\leq i \leq \ell -1}{\max} (n)_{2\ell-i} \int_{\mathbb X^i} \bigg( \int_{\mathbb X^{\ell-i}} h(x_1,\dots , x_\ell)
 d K^{\ell-i}(x_{i+1},\dots , x_\ell) \bigg)^2 d K^i(x_1,\dots,x_i)
$$
for $\ell\geq 2$, and put $r=0$ for $\ell=1$. Moreover for $n\geq 2\ell$, we define
$$
\tilde{S}=\frac{1}{\ell!} \sum_{(x_{1},...,x_{\ell})\in\beta_{n-2\ell,\neq}^{\ell}} h(x_{1},...,x_{\ell}).
$$

\begin{theorem}\label{U-stat-fin-bin}
Let $n\geq 2\ell$ and let $S$, $\lambda>0$, $r$ and $\tilde{S}$ be as above. Then,
\begin{equation}\label{fsxx1132-} 
d_{TV}(S,P_{\lambda}) 	\leq  (1 \wedge \lambda) \bigg( \frac{2^\ell r}{\ell! \lambda} + \frac{2 \ell^2 \lambda}{n} \bigg)
\quad \text{and} \quad
d_W(S, P_{\lambda}) 
 \leq (1.1437 \sqrt{\lambda} \wedge \lambda ) \bigg( \frac{2^\ell r}{\ell! \lambda} + \frac{2 \ell^2 \lambda}{n} \bigg).
\end{equation}
Moreover, for  all $m\in\mathbb{N}$,
	\begin{equation}\label{bin-bnd-0-}\begin{split}
			 \Big\vert \mathbb{P}(S=0) - e^{-\lambda} \Big\vert 
			 \leq \Bigg[\sum_{k=0}^{m-1} \bigg( \frac{\lambda}{k+1} \wedge \frac{k!}{\lambda^{k}} \bigg)  \mathbb{P}\big( \tilde{S} \leq k \big)  + \frac{m!}{\lambda^{m}} \Bigg] 
		\bigg( \frac{2^\ell r}{\ell! \lambda } + \frac{2 \ell^2 \lambda}{n} \bigg)
	\end{split}\end{equation}
	and for all $v\in\N$,
	\begin{align}\label{bin-bnd-0-II}
		&\vert \mathbb{P}(S\leq v) - \mathbb{P}(P_{\lambda} \leq v) \vert 
	\leq \bigg[\frac{(v+1)^2}{\lambda} + \mathbb{P} ( \tilde{S} \leq  v)\bigg] \bigg( \frac{2^\ell r}{\ell! \lambda} + \frac{2\ell^2 \lambda}{n} \bigg).  
	\end{align}
\end{theorem}
The bound on the Wasserstein distance in \eqref{fsxx1132-} slightly improves that in \cite[Theorem 7.1]{MR3502603} since it has a better order in $\lambda$. The bound for the total variation distance was also derived in \cite[Proposition 2]{MR3585403} by rewriting \cite[Theorem 2]{MR790624}. By means of \eqref{bin-bnd-0-}, one can study for some measurable symmetric function $g:\mathbb{X}^\ell \to \R$ the maximum (minumum) of $g(p)$ over all $p\in\beta_{n,\neq}^\ell$, which is called $U$-max-statistic ($U$-min-statistic). This is possible because  for any $u\in\R$, the probability that  $\max_{p\in\beta_{n,\neq}^\ell} g(p)$ is less than $u$ can be written as the probability that $\sum_{p\in\beta_{n,\neq}^\ell}\mathbf{1}\{g(p)\geq u\}$ equals $0$.  Limit theorems for $U$-max-statistics were considered in \cite{MR2466550}, yet without providing approximation results with respect to any distance; see also \cite{MR2394205}.  In contrast to these works,  \eqref{bin-bnd-0-} may lead to approximation results in the Kolmogorov distance; see Theorem \ref{Thm:min interp dist} in Subsection \ref{d12dx} and the discussion below it.    To the best of our  knowledge,  the last two inequalities presented in Theorem \ref{U-stat-fin-bin} have no analogue in the  literature.

From now on assume that $n\geq\ell$. Let $\chi$ be a point process of $\ell$ random points $X_1',\hdots, X_\ell'$ in $\mathbb{X}$ that are independent of $\beta_n$ and distributed such that
$$
\mathbb{P}( (X_1',\hdots,X_\ell') \in A ) = \frac{(n)_\ell}{\ell! \lambda} \int_{\mathbb{X}^\ell} \mathbf{1}\{ (x_1,\hdots,x_\ell)\in A \} h(x_1,\hdots,x_\ell) dK^\ell(x_1,\hdots,x_\ell)
$$
for all $A$ from the product $\sigma$-field $\mathcal{X}^\ell$.
Now we define
$$
S'= - h(X_1',\hdots,X_\ell') + \frac{1}{\ell!} \sum_{(x_{1},...,x_{\ell})\in(\beta_{n-\ell}\cup \chi)_{\neq}^{\ell}} h(x_{1},...,x_{\ell}) .
$$

\begin{proposition}\label{prop2}
	For all $n\geq \ell$ and $k\in\mathbb{N}$,
	$$
	k\mathbb{P}(S=k) = \lambda \mathbb{P}( S' =k-1).
	$$
\end{proposition}

\begin{proof}
We have that 
\begin{align*}
	k\mathbb{P}(S=k) = \mathbb{E}[k \mathbf{1}\{S=k\} ] = \frac{1}{\ell!}\mathbb{E} \sum_{(x_1,\hdots,x_\ell)\in\beta_{n,\neq}^\ell} h(x_1,\hdots,x_\ell) \mathbf{1}\{S=k\}.  
\end{align*}
Using the fact that for any measurable map $g:\mathbb{X}^u\times\mathbf{N}_{\mathbb{X}}\to[0,\infty)$ with $u\in\N$,
$$
\mathbb{E} \sum_{(x_1,\hdots,x_u)\in\beta_{n,\neq}^u} g(x_1,\hdots,x_u,\beta_n) = (n)_u \int_{\mathbb{X}^u} \mathbb{E}[g(x_1,\hdots,x_u,\beta_{n-u}+\sum_{i=1}^u \delta_{x_i})] dK^u(x_1,\hdots,x_u),
$$
we obtain
\begin{align*}
k\mathbb{P}(S=k) & = \frac{(n)_\ell}{\ell!} \int_{\mathbb{X}^\ell} h(x_1,\hdots,x_\ell) \mathbb{P}\bigg( \frac{1}{\ell!} \sum_{(y_1,\hdots,y_\ell)\in (\beta_{n-\ell}\cup\{x_1,\hdots,x_\ell\})^\ell_{\neq}} h(y_1,\hdots,y_\ell)=k \bigg) \\
& \qquad \qquad \times dK^\ell(x_1,\hdots,x_\ell) \\
& = \lambda \mathbb{P}(S'+h(X_1',\hdots,X_\ell')=k) = \lambda \mathbb{P}(S'=k-1),
\end{align*}
where we used $h(X_1',\hdots,X_\ell')=1$ in the last step. This concludes the proof.
\end{proof}

\begin{proof}[Proof of Theorem \ref{U-stat-fin-bin}]
Suppose $n\geq 2\ell$. Our goal is to apply Theorem \ref{Thmmm} with $Z=S'-S$, which satisfies the assumption \eqref{Mek} by Proposition \ref{prop2}. We define $s:\mathbf{N}_{\mathbb{X}}\to \mathbb{R}$ by
$$
s(\nu)=\frac{1}{\ell!} \sum_{(x_1,\hdots,x_\ell)\in \nu^\ell_{\neq}} h(x_1,\hdots,x_\ell)
$$
so that $S=s(\beta_{n})$ and $S'=s(\beta_{n-\ell}+\chi)-h(X_1',\hdots,X_\ell')$. By the monotonicity of $s$, we have
\begin{equation}\label{eqn:DecompositionBinomial}
\begin{split}
|Z| = |S'-S| & = |s(\beta_{n-\ell}+\chi)-h(X_1',\hdots,X_\ell')-s(\beta_{n-\ell}) - (s(\beta_n) - s(\beta_{n-\ell}))| \\
& \leq (s(\beta_{n-\ell}+\chi)-h(X_1',\hdots,X_\ell')-s(\beta_{n-\ell})) + s(\beta_n) - s(\beta_{n-\ell}) .
\end{split}
\end{equation}
Together with
\begin{equation}\label{eqn:DecompositionBinomial2}
\begin{split}
& s(\beta_{n-\ell}+\chi)-h(X_1',\hdots,X_\ell')-s(\beta_{n-\ell}) + s(\beta_n) - s(\beta_{n-\ell}) \\
& = s(\beta_{n-\ell}+\chi)-h(X_1',\hdots,X_\ell')-s(\beta_{n}) + 2(s(\beta_n) - s(\beta_{n-\ell})) = Z + 2 (s(\beta_n) - s(\beta_{n-\ell}))
\end{split}
\end{equation}
this implies
\begin{align*}
\mathbb{E}[|Z|] 
\leq 
\mathbb{E}[Z] + 2 \mathbb{E}[s(\beta_n) - s(\beta_{n-\ell})].
\end{align*}
From \eqref{eq: meanZ} in Remark \ref{remark1} we know that
$$
\mathbb{E}[Z] = \frac{1}{\lambda} (\mathrm{Var}(S) - \lambda)= \frac{1}{\lambda} (\mathbb{E}[S^2] - \lambda^2 - \lambda).
$$
Thus, it follows from \cite[Lemma 6.1]{MR3502603} and the definition of $r$ that
$$
\mathbb{E}[Z] \leq \frac{2^\ell r}{\ell! \lambda}.
$$
A straightforward computation shows that 
\begin{align*}
\mathbb{E}[s(\beta_n) - s(\beta_{n-\ell})] & = \bigg(1 - \frac{(n-\ell)_{\ell}}{(n)_\ell} \bigg) \lambda = \frac{(n)_\ell - (n-\ell)_\ell}{(n)_\ell} \lambda 
 \leq \frac{\ell^2 (n-1)_{\ell-1}}{(n)_\ell} \lambda = \frac{\ell^2 \lambda}{n}.
\end{align*}
Combining the previous estimates yields
$$
\mathbb{E}[|Z|] \leq  \frac{2^\ell r}{\ell! \lambda} +  \frac{2\ell^2 \lambda}{n}
$$
so that \eqref{fsxx1132-} follows from \eqref{fsxx}.

Let $k\in\N$ be fixed. Note that $S\geq s(\beta_{n-\ell})$ and $S'\geq s(\beta_{n-\ell})$. If $Z\geq 0$, this implies
$$
\mathbf{1}\{S-Z_{-}\leq k \} = \mathbf{1}\{S\leq k \} \leq \mathbf{1}\{s(\beta_{n-\ell})\leq k \}.
$$
For $Z\leq 0$ we obtain
$$
\mathbf{1}\{S-Z_{-}\leq k \} = \mathbf{1}\{S+Z\leq k \} = \mathbf{1}\{S'\leq k \} \leq \mathbf{1}\{s(\beta_{n-\ell})\leq k \}.
$$
Combing the two cases leads to
$$
\mathbf{1}\{S-Z_{-} = k \} \leq \mathbf{1}\{S-Z_{-}\leq k \} \leq \mathbf{1}\{s(\beta_{n-\ell})\leq k \}.
$$
Together with \eqref{eqn:DecompositionBinomial} we obtain
\begin{equation}\label{eqn:|Z|Ind}
\begin{split}
& \mathbb{E}[ |Z| \mathbf{1}\{S-Z_{-}=k\}] \\
& \leq \mathbb{E}[ \mathbf{1}\{s(\beta_{n-\ell})\leq k \} ( s(\beta_{n-\ell}+\chi)-h(X_1',\hdots,X_\ell')-s(\beta_{n-\ell}) + s(\beta_n) - s(\beta_{n-\ell}) ) ].
\end{split}
\end{equation}
For $u\in\{1,\hdots,\ell-1\}$ and $g:\mathbb{X}^u\to[0,\infty)$, we have
\begin{equation}\label{eqn:CorrelationUstatisticsBinomial}
\begin{split}
& \mathbb{E} [\mathbf{1}\{s(\beta_{n-\ell})\leq k \} \sum_{(x_1,\hdots,x_u)\in\beta_{n-\ell,\neq}^u} g(x_1,\hdots,x_u) ]\\
& = (n-\ell)_u \int_{\mathbb{X}^u} \mathbb{P}(s(\beta_{n-\ell-u}+\sum_{i=1}^u \delta_{x_i})\leq k) g(x_1,\hdots,x_u) dK^u(x_1,\hdots,x_u) \\
& \leq \mathbb{P}(s(\beta_{n-2\ell})\leq k) (n-\ell)_u \int_{\mathbb{X}^u} g(x_1,\hdots,x_u) dK^u(x_1,\hdots,x_u) \\
& = \mathbb{P}(s(\beta_{n-2\ell})\leq k) \mathbb{E} \sum_{(x_1,\hdots,x_u)\in\beta_{n-\ell,\neq}^u} g(x_1,\hdots,x_u),
\end{split}
\end{equation}
where the inequality follows from the monotonicity of $s$. Because of 
$$
s(\beta_{n-\ell}+\chi)-h(X_1',\hdots,X_\ell')-s(\beta_{n-\ell}) = \sum_{u=1}^{\ell-1} \sum_{(x_1,\hdots,x_u)\in\beta_{n-\ell,\neq}^u} \tilde{h}_u(x_1,\hdots,x_u; \chi)
$$
and
$$
s(\beta_n) - s(\beta_{n-\ell}) =  \sum_{u=1}^{\ell-1} \sum_{(x_1,\hdots,x_u)\in\beta_{n-\ell,\neq}^u} \overline{h}_u(x_1,\hdots,x_u; \beta_n\setminus\beta_{n-\ell})
$$ 
with suitable functions $\tilde{h}_u$ and $\overline{h}_u$, $u\in\{1,\hdots,\ell-1\}$, we can rewrite the second factor on the right-hand side of \eqref{eqn:|Z|Ind} as sum of $U$-statistics with respect to $\beta_{n-\ell}$. Now an application of  \eqref{eqn:CorrelationUstatisticsBinomial} and \eqref{eqn:DecompositionBinomial2} yield
\begin{align*}
& \mathbb{E}[ |Z| \mathbf{1}\{S-Z_{-}=k\}] \\
& \leq  \mathbb{P}(s(\beta_{n-2\ell})\leq k) \mathbb{E}[ s(\beta_{n-\ell} + \chi)-h(X_1',\hdots,X_\ell')-s(\beta_{n-\ell}) + s(\beta_n) - s(\beta_{n-\ell})  ] \\
& = \mathbb{P}(s(\beta_{n-2\ell})\leq k) \big( \mathbb{E}[Z] + 2 \mathbb{E}[s(\beta_n) - s(\beta_{n-\ell})] \big) .
\end{align*}
Bounding the second factor on the right-hand side as above leads to
$$
\mathbb{E}[ |Z| \mathbf{1}\{S-Z_{-}=k\}] \leq \mathbb{P}(s(\beta_{n-2\ell})\leq k)  \bigg( \frac{2^\ell r}{\ell! \lambda} +  \frac{2\ell^2 \lambda}{n} \bigg).
$$
Thus, \eqref{bin-bnd-0-} and \eqref{bin-bnd-0-II} are immediate consequences of \eqref{rrrf} and \eqref{eq:boundOnv}.
\end{proof}

\subsection{$U$-statistics of Poisson processes}\label{U-stt-P-prf}
In this subsection, we study the Poisson approximation of $U$-statistics, where one sums over all $\ell$-tuples of distinct points of a Poisson process instead of those of a binomial point process as in the previous subsection. In this case, the summation can run over infinitely many $\ell$-tuples.
As the results for $U$-statistics with binomial input in Subsection \ref{U-stt-B-prf}, the theory developed herein permits to study extreme value problems arising in stochastic geometry. For example, in the next subsection, we employ our main result for $U$-statistics with Poisson input to investigate the limiting behavior of the minimum inter-point distance between the points of a Poisson process in $\mathbb{R}^d$. 

Let $(\mathbb X,\mathcal{X})$ be a measurable space and let $\eta$ be a Poisson process with a $\sigma$-finite intensity measure $L$ on $\mathbb{X}$. For a fixed $\ell\in\N$ and a symmetric measurable function $h:\mathbb{X}^\ell\to\{0,1\}$ that is integrable with respect to $L^\ell$ we consider the $U$-statistic 
$$
S=\frac{1}{\ell!} \sum_{(x_{1},...,x_{\ell}) \in \eta_{\neq}^{\ell}} h(x_{1}, \ldots, x_{\ell}),
$$
where $\eta_{\neq}^{\ell}$ denotes the set of all $\ell$-tuples of distinct points of $\eta$. It follows from the multivariate Mecke formula that
$$
\lambda:=\mathbb{E}[S]=\frac{1}{\ell!} \int_{\mathbb{X}^\ell} h(x_1,\hdots,x_\ell) dL^\ell(x_1,\hdots,x_\ell).
$$
We define
$$
r= \underset{1\leq i \leq \ell -1}{\operatorname{max}} \int_{\mathbb X^i} \bigg( \int_{\mathbb X^{\ell-i}} h(x_1,\dots , x_\ell)  d L^{\ell-i}(x_{i+1},\dots , x_\ell) \bigg)^2 d L^i(x_1,\dots,x_i)
$$
for $\ell\geq 2$, and put $r=0$ for $\ell=1$. The expression $r$ is used to quantify the accuracy of the Poisson approximation for $S$ and it is the analogue of $r$ given in Subsection \ref{U-stt-B-prf} for binomial $U$-statistics.  

\begin{theorem}\label{U-stat-fin}
	 Let $S$, $\lambda>0$ and $r$ be as above. Then,
	\begin{align} \label{eq: wassU-sta-P}  
		d_{TV}(S, P_{\lambda}) 
		\leq \bigg(1 \wedge \frac{1}{\lambda} \bigg) \frac{2^{\ell} r}{\ell!}
\quad \text{and} \quad
		d_W(S, P_{\lambda}) 
		& \leq \bigg(1 \wedge \frac{1.1437 }{\sqrt {\lambda}}\bigg)\frac{2^{\ell} r}{\ell!}.
	\end{align}
	Moreover, for all $m\in\mathbb{N}$,
	\begin{equation}\begin{split}\label{rrrftq1}
			0 
			 \leq \mathbb{P}(S=0) - e^{-\lambda} 
			 \leq \Bigg[\sum_{k=0}^{m-1} \bigg( \frac{1}{k+1} \wedge \frac{k!}{\lambda^{k+1}} \bigg)  \mathbb{P}\big( S \leq k \big)  + \frac{m!}{\lambda^{m+1}} \Bigg] \frac{2^{\ell}r}{\ell!}  
	\end{split}\end{equation}
	and for all $v\in\N$, 
	\begin{align}\label{eqn: cumdistrbPoiss-U}
		\vert \mathbb{P}(S\leq v) - \mathbb{P}(P_{\lambda} \leq v) \vert 
		\leq \Bigg[\frac{(v+1)^2}{\lambda} + \mathbb{P} ( S \leq  v)\Bigg]\frac{2^{\ell}  r}{\ell! \lambda} .  
	\end{align}
\end{theorem}
The result for the total variation distance in \eqref{eq: wassU-sta-P} was shown in \cite[Proposition 1]{MR3585403}, which improved \cite[Proposition 4.1]{MR2971726}, and in \cite[Section 8]{MR3780386}. The bound for the Wasserstein distance in \eqref{eq: wassU-sta-P} was also derived in \cite[Section 8]{MR3780386} and has a slightly better order in $\lambda$ than that in \cite[Theorem 7.1]{MR3502603}. To the best of our  knowledge,  the other inequalities presented in Theorem \ref{U-stat-fin} have no analogues in the  literature. 

\begin{proof}[Proof of Theorem \ref{U-stat-fin}]
We follow a similar approach as in the proof of Theorem \ref{U-stat-fin-bin}. 
For $\ell =1,$ Theorem \ref{U-stat-fin} is a direct consequence of \cite[Theorem 5.1]{MR3791470}, whence we assume $\ell \geq 2$ from now on. 

Let $\chi$ be a point process of $\ell$ random points $X_1',\hdots,X_\ell'$ that are independent of $\eta$ and distributed according to
$$
\mathbb{P}((X_1',\hdots,X_\ell')\in A) = \frac{1}{\ell! \lambda} \int_{\mathbb{X}^\ell} \mathbf{1}\{(x_1,\dots,x_\ell)\in A \} h(x_{1},\dots,x_{\ell}) d L^\ell(x_1,\dots,x_\ell)
$$
for $A\in\mathcal{X}^\ell$. We define
$$
S'=-h(X_1',\hdots,X_\ell') + \frac{1}{\ell!} \sum_{(x_1,\hdots,x_\ell)\in(\eta\cup\chi)^k_{\neq}} h(x_1,\hdots,x_\ell).
$$
For $k\in\N$ the multivariate Mecke formula implies that
\begin{align*}
& k \mathbb{P}(S=k)  = \mathbb{E}[ S \mathbf{1}\{S=k\} ] \\
& = \frac{1}{\ell!} \mathbb{E} \sum_{(x_1,\hdots,x_\ell)\in\eta^\ell_{\neq}} h(x_1,\hdots,x_{\ell}) \mathbf{1}\bigg\{ \frac{1}{\ell!} \sum_{(y_1,\hdots,y_\ell)\in\eta^\ell_{\neq}} h(y_1,\hdots,y_{\ell})=k  \bigg\} \\
& = \frac{1}{\ell!} \int_{\mathbb{X}^\ell} h(x_1,\hdots,x_\ell) \mathbb{P}\bigg( \frac{1}{\ell!} \sum_{(y_1,\hdots,y_\ell)\in(\eta \cup \{x_1,\dots,x_\ell\})^\ell_{\neq}} h(y_1,\hdots,y_{\ell}) = k \bigg) dL^\ell(x_1,\hdots,x_\ell) \\
& = \lambda \mathbb{P}(S'+h(X_1',\hdots,X_\ell')=k) = \lambda \mathbb{P}(S'=k-1),
\end{align*}
where we used $h(X_1',\hdots,X_\ell')=1$ in the last step.
Thus, we see that $S$ satisfies  the hypothesis of Theorem \ref{Thmmm} with $Z=S'-S\geq 0$.

Next we compute the expressions on the right-hand sides of the bounds in Theorem \ref{Thmmm}.
Let $k\in\N$ be fixed. Define $s(\nu)=\frac{1}{\ell!} \sum_{(x_1,\hdots,x_\ell)\in\nu^{\ell}_{\neq}} h(x_1,\hdots,x_\ell)$ for $\nu\in\mathbf{N}_{\mathbb{X}}$ and note that $S=s(\eta)$. Since
	$$
s(\nu+\chi+\delta_x)-s(\nu+\delta_x) \geq s(\nu+\chi)-s(\nu) \quad \text{and} \quad	\mathbf{1}\{s(\nu +\delta_x )\leq k\}\leq  \mathbf{1}\{s(\nu)\leq k\}
	$$
	for all $\nu\in\mathbf{N}_{\mathbb{X}}$ and $x\in\mathbb{X}$, by \cite[Theorem 20.4]{MR3791470} we obtain
	$$
	\mathbb{E}\big[ Z \mathbf{1}\{S\leq k\}\big]\leq \mathbb{E}[ Z]\mathbb{P}(S\leq k).
	$$
	Together with $Z\geq 0$, we have
	\begin{equation}\label{pcaxx}
		\mathbb{E}[ |Z|  \mathbf{1}\{S-Z_{-}=k\}] = \mathbb{E}[ Z  \mathbf{1}\{S=k\}]
		\leq \mathbb{E}[ Z \mathbf{1}\{S\leq k\}]  
		\leq \mathbb{E}[Z] \mathbb{P}(S\leq k).
	\end{equation}	
	Furthermore, from Remark \ref{remark1}-$(ii)$ it follows that 
	\begin{align}\label{eq: P-U-ident-meanZ}
	\mathbb{E}[Z]  
	= \frac{1}{\lambda}\big\{ \mathrm{Var}(S) - \lambda\big\} = \frac{1}{\lambda} \big\{ \mathbb{E}[S^2] - \lambda^2 - \lambda \big\}.
\end{align}
	Then, from $Z\geq 0$ and \cite[Lemma 6.1]{MR3502603} we deduce
	\begin{align}\label{eq: P-U-stat-bnd-Z}
		\mathbb{E}[|Z|] = \mathbb{E}[Z] \leq  \sum_{i=1}^{\ell-1} \frac{1}{\ell!\lambda} \binom{\ell}{i} r
		\leq \frac{2^{\ell}}{\ell! \lambda } r.
	\end{align}
Finally, combining this bound with \eqref{fsxx} shows \eqref{eq: wassU-sta-P}, while \eqref{rrrf} and \eqref{eq:boundOnv} together with \eqref{pcaxx} and \eqref{eq: P-U-stat-bnd-Z} lead to \eqref{rrrftq1}, where the first inequality is a consequence of Corollary \ref{prpbound1} a), and \eqref{eqn: cumdistrbPoiss-U}.
\end{proof}

\subsection{The  distances between the points of a Poisson  process}\label{d12dx}
We consider random points in $\R^d$ distributed according to a Poisson process. For any pair of these points with the midpoint in a bounded measurable set $W\subset\R^d$, we take a transform of the Euclidean distance, and we study the Poisson approximation for the number of times that these quantities belong to a certain range of values. More importantly, we consider the exponential approximation for a transform of the minimal distance between pairs of points with midpoint in $W$.

For convenience, we assume $\lambda_d(W)=1$; nonetheless, the following arguments are valid for every $W$ with  a positive and finite volume.  Let $\eta_{t}$ be a Poisson process on $\mathbb{R}^d$ with intensity measure $t\lambda_d$, $t>0$, where we denote by $\lambda_d$ the $d$-dimensional Lebesgue measure. Define
\begin{align*}
	\xi_t 
	& =\frac{1}{2}\sum_{(x,y)\in\eta_{t,\neq}^{2}}\mathbf{1}\Big\{\frac{x+y}{2}\in W\Big\}\delta_{2^{-1}t^2k_d\Vert x-y\Vert^d  }, \,\,\, \quad t>0,
	\\
	Y_t 
	& =  \min_{(x,y)\in \eta_{t,\neq}^{2}: \frac{x+y}{2}\in W}  2^{-1} t^2 k_d \Vert x-y \Vert^d , \quad t>0,
\end{align*}
where $\delta_z$ stands for the Dirac measure at $z\in\R$ and $\Vert \cdot \Vert$ denotes the Euclidean norm.

\begin{theorem}\label{Thm:min interp dist}
	Let $\xi_t$ and $Y_t$ be as above for $t>0$. Let $\gamma$ be a Poisson process on $[0,\infty)$ with the restriction of the Lebesgue measure to $[0,\infty)$ as intensity measure. Then for all $u\geq 0$ and all measurable $B\subset [0,u]$, 
	\begin{align}\label{dq:TV-conv_intrpt}
		&  d_{TV}(\xi_t(B),\gamma(B)) 
			\leq  (1 \wedge u ) \frac{8u}{t}
	\end{align}
	and
\begin{align}\label{applsei-eq2}
		0 
		\leq\mathbb{P}(Y_t>u) - e^{-u}
		\leq \frac{81}{t}.
\end{align} 

\end{theorem}

The minimal distance between the points of a Poisson process was also considered in \cite{MR3252817,MR3502603,MR2971726,MR3585403}, sometimes formulated as minimal edge length of the random geometric graph or the minimal inradius of a Poisson-Voronoi tessellation. The important achievement of Theorem \ref{Thm:min interp dist} is that a rate of convergence for the Kolmogorov distance is provided in \eqref{applsei-eq2}.  So far it was only possible to prove bounds on the difference between $\mathbb P( Y_t>u)$ and $e^{-u}$ that depend on $u>0$ (see e.g.\ \cite[Theorem 2.4]{MR2971726} or \cite[Corollary 3]{MR3585403}).

In the works mentioned above all pairs of points are considered such that one or both points belong to $W$. Our approach, where we only require that the midpoint of the points is in $W$, can be extended to this different way of counting, but one might get additional terms in the bounds since $\mathbb{E}[\xi_t([0,u])]$ is not necessarily $u$ due to boundary effects.

In \cite{MR2971726,MR3585403}, beside Poisson approximation results for the number of interpoint distances below a given threshold it was shown that the point process of rescaled interpoint distances converges weakly to a Poisson process.  By \eqref{dq:TV-conv_intrpt} and \cite[Theorem 16.16]{MR1876169}, we can also deduce that $\xi_t$ converges weakly to $\gamma$ as $t\to\infty$.

The related problem of small distances between the points of a binomial point process was first studied in \cite{MR511059}. Because of the similarity to   Theorem \ref{U-stat-fin}, we believe that by applying Theorem \ref{U-stat-fin-bin} it is possible to prove a similar result to Theorem \ref{Thm:min interp dist} for an underlying binomial point process.

By using in the proof of Theorem \ref{Thm:min interp dist} the corresponding bound of Theorem \ref{U-stat-fin} for the Wasserstein distance, one can obtain the counterpart of \eqref{dq:TV-conv_intrpt} for the Wasserstein distance with a different power in $u$ and the same rate of convergence in $t$.

\begin{proof}[Proof of Theorem \ref{Thm:min interp dist}]
First, we show that the intensity measure of the point process $\xi_t$ is the restriction of the Lebesgue measure to $[0,\infty)$.
	Let $v_t=\big(\frac{2u}{ k_d t^2}\big)^{1/d}$. The change of variable $z=\frac{x+y}{2}$ yields
	\begin{align}
		\mathbb{E}[\xi_t([0,u])]
		& =\frac{t^2}{2}\int_{\mathbb{R}^d}\int_{\mathbb{R}^d}\mathbf{1}\Big\{\frac{x+y}{2}\in W\Big\}\mathbf{1}\{\Vert x-y\Vert \leq v_t\} \,dydx
		\notag
		\\
		& =2^{d-1} t^2\int_{\mathbb{R}^d}\int_{\mathbb{R}^d}\mathbf{1}\big\{z\in W\big\}\mathbf{1}\{2\Vert x-z\Vert \leq v_t\} \,dz dx
		\notag
		\\
		& =2^{d-1} t^2\int_{W}\int_{\mathbb{R}^d}\mathbf{1}\{2\Vert x-z\Vert \leq v_t\} \, dx dz
		= u.
		\notag
	\end{align}
	For   $B\subset [0,u]$  with $u>0$	define 
	$$
		r_t(B)=  t  \int_{\R^d} \bigg( t \int_{\R^d} \mathbf{1}\Big\{\frac{x+y}{2}\in W\Big\}\mathbf{1}\{2^{-1} t^2 k_d \Vert x-y \Vert^d\in B\}\,    d y \bigg)^2 d x.
	$$
Again from the change of variable  $z=\frac{x+y}{2},$  it follows that  
	\begin{align*}
		r_t(B)
	&\leq r_t([0,u])
	 = 2^{2d} t^{3} \int_{\R^d} \bigg(  \int_{W} \mathbf{1}\{2 \Vert x - z \Vert \leq v_t\}   \, d z \bigg)^2 d x 
	 \\
	 & \leq 2^{2d} t^{3} \int_{\R^d} \bigg( \int_{\R^d} \mathbf{1}\{2 \Vert x - z \Vert \leq v_t\} \, d z \int_{W}  \mathbf{1}\{2 \Vert x - \tilde{z}\Vert \leq v_t\}\, d \tilde{z}  \bigg) d x
	 \\
	 & =2^{d+1} u t \int_{\R^d} \int_{W}  \mathbf{1}\{2 \Vert x - \tilde{z}\Vert \leq v_t\} \,  d \tilde{z} dx \\
	 &= 2^{d+1} u t \int_{W} \int_{\R^d}  \mathbf{1}\{2 \Vert x - \tilde{z}\Vert \leq v_t\} \, d x d \tilde{z} = \frac{4u^2}{t} .
	\end{align*}
	Therefore \eqref{eq: wassU-sta-P} in Theorem \ref{U-stat-fin} with $h(x,y)=\mathbf{1}\big\{\frac{x+y}{2}\in W\big\}\mathbf{1}\{2^{-1}t^2k_d\Vert x-y\Vert^d  \in B\}$ yields for $B\subset [0, u]$ that
	\begin{align*}
		d_{TV}(\xi_t(B),\gamma (B))  &\leq \bigg(1 \wedge \frac{1}{u} \bigg) 2 r_t(B) \leq (1 \wedge  u ) \frac{ 8 u }{t} .
	\end{align*}

From \eqref{eq: P-U-ident-meanZ} and \eqref{eq: P-U-stat-bnd-Z} in the proof of Theorem \ref{U-stat-fin} with $S=\xi_t([0,u]),$ $r=r_t([0,u])$ and $h$ as above, we know that
$$
\mathrm{Var}(\xi_t([0,u])) \leq \mathbb{E}[\xi_t([0,u])] + 2r_t([0,u]) = u+\frac{8u^2}{t}.
$$
Thus it follows from the Chebyshev inequality that
$$
\mathbb{P}(\xi_t([0,u])=0) \leq \mathbb{P}(|\xi_t([0,u]) - u|\geq u) \leq \frac{\mathrm{Var}(\xi_t([0,u]))}{u^2} = \frac{1}{u}+\frac{8}{t}.
$$
Together with \eqref{rrrftq1} in Theorem  \ref{U-stat-fin} with $m=1$ and straightforward arguments, this leads to
	\begin{align}
		0 
		\leq \mathbb{P}(\xi_t([0,u])=0) - e^{-u} 
		& = \mathbb{P}(Y_t>u) - e^{-u} 
	\leq \bigg[\frac{1}{u} \mathbb{P} (\xi_t([0,u]) =0 ) + \frac{1}{ u^{2}} \bigg] \frac{8u^2}{t}
	\label{eq: min-dist-CD-lowerbnd}
	\\
	& \leq  \bigg( \frac{1}{u^2} + \frac{8}{ut} + \frac{1}{u^2} \bigg) \frac{8u^2}{t} = \frac{16}{t} + \frac{64 u}{t^2}
	\notag
	\end{align}
so that
$$
\sup_{u\in [0, t]} | \mathbb{P}(Y_t> u) - e^{-u} | \leq \frac{80}{t}.
$$
Thus, we have
$$
\mathbb{P}(Y_t> t) \leq \frac{80}{t} + e^{-t} \leq \frac{81}{t} 
$$
and
$$
\sup_{u\in [0, \infty)} | \mathbb{P}(Y_t> u) - e^{-u} | \leq \max\big\{ \sup_{u\in [0, t]} | \mathbb{P}(Y_t> u) - e^{-u} |, \mathbb{P}(Y_t> t), e^{-t} \big\} \leq \frac{81}{t},
$$
which combined with the left-hand side of \eqref{eq: min-dist-CD-lowerbnd} completes the proof.
\end{proof}

\subsection{Long head runs}\label{succe}
 Consider $n$ independent and identically distributed Bernoulli random variables. A $k$-head run is defined as an uninterrupted sequence of $k$ successes, where $k$ is a positive integer. For example, for $k = 1$,  one simply studies the successes, while for $k = 2$, one considers the occurrence of two  consecutive successes in a row. Several authors have investigated the number of $k$-head runs in a sequence of Bernoulli random variables; see e.g.\ the book \cite{MR1882476}. In this subsection, we discuss the Poisson approximation of the number of non-overlapping $k$-runs among $n$ i.i.d.\  Bernoulli random variables, denoted by $S_{n,k}$.   In particular, we obtain an explicit  bound  on the pointwise difference between the cumulative distribution functions of $S_{n,k}$ and $P_{\mathbb{E}[S_{n,k}]}$ that is independent from the number $k$ of required successes in a row.

Let $k\in\N$ and  $X_j$, $j\in\mathbb{N}_0,$ be a sequence of independent and Bernoulli distributed random variables with parameter $0< p\leq 1/2$. We denote by $I^{(i)}$ with $i\in\N_0$ the random variable
\begin{align*}
I^{(i)}=\mathbf{1}\{X_{i-1}=0, X_i=1,\dots,X_{i+k-1}=1\},
\end{align*}
where $X_{-1}=0$. For $k\leq n$ the number $S_{n,k}$ of non-overlapping $k$-runs in $X_0,\dots ,X_{n-1}$ is given by
\begin{align}\label{S}
S_{n,k}
=\sum_{i=0}^{n-k} I^{(i)}.
\end{align}

\begin{theorem}\label{Thm-S}
Let  $S_{n,k}$  be the random variable given by \eqref{S} with $k,n\in\N, k\leq n$. Then, 
\begin{equation}\label{TV-S}\begin{split}
 d_{TV}\big(S_{n,k},P_{\mathbb E [S_{n,k}]}\big) 
\leq (2k+1)\big( 1 \wedge \mathbb E [S_{n,k}]\big) p^k.
\end{split}\end{equation}
Moreover, for  $v\in\N_0$ and $n\geq 2$,
\begin{equation}\label{eq:sup-k-heads-run}\begin{split}
\max_{1\leq k \leq n} \vert  \mathbb{P}(S_{n,k}\leq v) - \mathbb{P}(P_{\mathbb{E}[S_{n,k}]}\leq v) \vert 
\leq 40(v+2)^2 \frac{\log n}{n} .
\end{split}\end{equation}
\end{theorem}

The bound \eqref{TV-S} was shown in \cite[Corollary 15]{MR3992498} as a consequence of \cite[Theorem 1]{MR972770}. The Poisson approximation for $S_{n,k}$  is also investigated in e.g.\ \cite{MR1092983,MR1163825,MR1133732,MR1334894}. The explicit bound in \eqref{eq:sup-k-heads-run} on the pointwise difference between the cumulative distribution functions of $S_{n,k}$ and $P_{\mathbb{E}[S_{n,k}]}$ does not depend on the number $k$ of required successes in a row. Hence,  \eqref{eq:sup-k-heads-run} improves \cite[Corollary 3.23]{MR2933280} and \cite[Corollary 16]{MR3992498} because we found an explicit bound. Furthermore, since the proof of Theorem \ref{Thm-S} is based on Theorem \ref{Thmmm}, by applying the second inequality of \eqref{fsxx} in Theorem \ref{Thmmm},  it is possible to attain a bound on the  Wasserstein distance between $S_{n,k}$ and $P_{\mathbb E [S_{n,k}]}$.

For the proof of Theorem \ref{Thm-S} we define
$$
U_\ell  
= \sum_{i= 0\vee (\ell-k)}^{(n-k) \wedge (\ell+k)} I^{(i)}, \quad \ell=0,\dots, n-k,
$$
where $a\vee b = \max\{a,b\}$ for any $a,b\in\R$, and let $Y$ be a random variable  independent from $X_j, j\in\N_0,$ and with distribution given by 
$$
\mathbb{P}(Y=\ell) = \frac{\mathbb{E}[I^{(\ell)}]}{\mathbb E [S_{n,k}]}, \quad  \ell= 0,\dots, n-k.
$$
The next proposition is derived by a standard construction of size-bias couplings (see e.g.\ \cite[Corollary 3.24]{MR2861132}). 
\begin{proposition}\label{size-bias-S-prop}
Let $k,n\in\N$ with $k\leq n$. For any $m\in\N$,
\begin{align*}
m\mathbb{P}(S_{n,k}=m)
= \mathbb E [S_{n,k}]\mathbb{P} (S_{n,k}- U_{Y}=m-1 ) .
\end{align*}
\end{proposition}

\begin{proof}
Let  $\ell\in \{0,\dots, n-k\}$ and $m\in\N$ be fixed. Then, we have
$$
\mathbb{E}[ I^{(\ell)} \mathbf{1}\{S_{n,k}- I^{(\ell)}=m-1\}] 
=\mathbb{E}[I^{(\ell)} \mathbf{1}\{S_{n,k}- U_{\ell}=m-1\}].
$$
Since $I^{(\ell)}$ and $S_{n,k} - U_\ell$ are independent, it follows  that
\begin{align*}
m\mathbb{P}(S_{n,k}=m)
& =  \sum_{\ell=0}^{n-k}\mathbb{E}[ I^{(\ell)} \mathbf{1}\{S_{n,k}=m\}]
= \sum_{\ell=0}^{n-k} \mathbb{E}[ I^{(\ell)} \mathbf{1}\{S_{n,k}- I^{(\ell)}=m-1\}] 
\\
&=\sum_{\ell=0}^{n-k} \mathbb{E}[I^{(\ell)}]\mathbb{P}(S_{n,k}- U_{\ell}=m-1) 
= \mathbb E [S_{n,k}] \mathbb{P}(S_{n,k}- U_{Y}=m-1) ,
\end{align*}
which concludes the proof.
\end{proof}
\begin{remark}
Since $U_Y\geq 0$, from Corollary \ref{prpbound1} b) it follows that $\mathbb{P}(S_{n,k}=0)\leq e^{-\mathbb{E}[S_{n,k}]}$. Thus, straightforward calculations imply that
$$
\mathbb{P}(S_{n,k}=0)\leq \exp\big(-(n-k+1)p^k(1-p)\big).
$$
\end{remark}
\begin{proof}[Proof of Theorem \ref{Thm-S}]
From \eqref{fsxx} in Theorem \ref{Thmmm} and Proposition \ref{size-bias-S-prop}, it follows that 
$$
 d_{TV}(S_{n,k},P_{\mathbb E [S_{n,k}]})
\leq  ( 1 \wedge \mathbb{E}[S_{n,k}]   )\mathbb{E}[ U_{Y}] 
 \leq (2k+1) ( 1 \wedge \mathbb{E}[S_{n,k}]   ) p^k,
$$
where we used $\mathbb{E}[U_{\ell}]\leq (2k+1)p^k$ for $\ell=0,\dots ,n-k$ in the last step. This proves   \eqref{TV-S}.

Let $n\geq 2$ be fixed. Since $(2k+1)p^k,k\geq 1$, is decreasing in $k$ for any $p\leq 1/2$,   by \eqref{TV-S} we deduce for $k\geq 2\log n $  that
\begin{align}\label{eq: bnd-over-k-large-k}
	\vert  \mathbb{P}(S_{n,k}\leq v) - \mathbb{P}(P_{\mathbb{E}[S_{n,k}]}\leq v) \vert  \leq (2 k +1 )p^k \leq (4\log n  +1 )2^{-2\log n} \leq \frac{4\log n +1 }{n}.
\end{align}
Let $k< 2\log n $. From \eqref{rrrf} in Theorem \ref{Thmmm} with $m=1$ for $v=0$ and \eqref{eq:boundOnv} in Theorem \ref{Thmmm} for $v\in\N$, it follows that
\begin{align}
	\vert  \mathbb{P}(S_{n,k}\leq v) - \mathbb{P}(P_{\mathbb{E}[S_{n,k}]}\leq v) \vert 
	& \leq \frac{(v+1)^2\mathbb{E}[ U_Y ] }{\mathbb{E}[S_{n,k}]} + \mathbb{E}[  U_Y \mathbf{1}\{ S_{n,k}-  U_Y \leq v  \} ] 
	 .
	\label{eq: intemidiat-bnd-on-less-than-v}
\end{align}
From $0 \leq U_{\ell} \leq 2$ for $\ell\in\{0,\hdots,n-k\}$ and the definition of $Y$ it follows that
\begin{align*}
	 \mathbb{E}[U_{Y} \mathbf{1}\{S_{n,k}-U_Y\leq v\}]  
	& \leq \mathbb{E}[U_{Y} \mathbf{1}\{S_{n,k} \leq v+2\}] = \mathbb{E} \sum_{\ell=0}^{n-k} \frac{\mathbb{E}[I^{(\ell)}]}{\mathbb{E}[S_{n,k}]} \mathbb{E}[U_{\ell} \mathbf{1}\{S_{n,k} \leq v+2\}] \\
	& \leq \frac{p^k}{\mathbb{E}[S_{n,k}]} \mathbb{E} \sum_{\ell=0}^{n-k} \sum_{i= 0\vee (\ell-k)}^{(n-k) \wedge (\ell+k)} I^{(i)} \mathbf{1}\{S_{n,k} \leq v+2\}.
\end{align*}
Thus, by the inequality
$$ 
\sum_{\ell=0}^{n-k} \sum_{i= 0\vee (\ell-k)}^{(n-k) \wedge (\ell+k)} a_i\leq (2k+1)\sum_{m=0}^{n-k} a_m,\quad   a_0,\dots,a_{n-k}\geq 0,
$$
 we obtain
\begin{align*}
   \mathbb{E}[U_{Y} \mathbf{1}\{S_{n,k}-U_Y\leq v\}]	
	&\leq  \frac{(2k+1) p^k}{\mathbb{E}[S_{n,k}]} \mathbb{E}[ S_{n,k} \mathbf{1}\{S_{n,k} \leq v+2\} ]\leq  \frac{(2k+1) p^k (v+2)}{\mathbb{E}[S_{n,k}]}. 
\end{align*}
Together with \eqref{eq: intemidiat-bnd-on-less-than-v} and the inequalities 
$$
\mathbb{E}[ S_{n,k} ]  \geq (n-k +1) p^k/2 \quad \text{ and } \quad \mathbb{E}\big[ U_Y \big] \leq (2k +1)p^k,
$$
this shows for $k<2\log n$ and $n>4 \log n $ that 
\begin{align*}
	\vert  \mathbb{P}(S_{n,k}\leq v) - \mathbb{P}(P_{\mathbb{E}[S_{n,k}]}\leq v) \vert 
		& \leq \frac{2(v+1)^2 (2k +1)}{ n-k +1  } + \frac{ 2(v+2)(2k+1)}{n-k+1}
		\\
		& \leq \frac{4 (v+2)^2( 4 \log n +1)}{ n - 2\log n  }\leq \frac{40(v+2)^2 \log n}{n},
\end{align*}
where we used the inequalities $4 \log n + 1 \leq 5\log n$ and  $ n - 2\log n  \geq n/2$ for $n>4\log n $ in the last step.  Combining  this and \eqref{eq: bnd-over-k-large-k} establishes \eqref{eq:sup-k-heads-run} for $n> 4\log n $. In conclusion, note that $n> 4\log n$ for $n>10$, and for $2\leq n\leq 10$, the right-hand side of \eqref{eq:sup-k-heads-run}  is greater than $1$. Thus, \eqref{eq:sup-k-heads-run} holds for all $n\geq 2$.
\end{proof}

\subsection{Minimal circumscribed radii of Poisson-Voronoi tessellations}\label{circ-rad-proof}
In this subsection, we consider  circumscribed radii of stationary Poisson-Voronoi tessellations.  The aim is to continue the work started in \cite{MR3252817} by proving that the Kolmogorov distance between a transform of the minimal circumscribed radius and a Weibull random variable converges to $0$ at a rate of $1/t^{1/(d+1)}$ when the intensity $t$ of the underlying Poisson process goes to infinity.

For any locally finite counting measure $\nu$ on $\R^d$, we denote by $N(x,\nu)$ the Voronoi cell with nucleus $x\in\R^d$ generated by $\nu + \delta_x$, that is
$$
N(x,\nu)
=\left\{y\in\mathbb{R}^d\,:\,\Vert y-x\Vert  \leq   \Vert  y-x'\Vert ,x\neq x'\in \nu\right\}.
$$
  Voronoi tessellations, i.e., tessellations consisting of Voronoi cells $N(x,\nu)$, $x\in\nu$, arise in different fields such as biology \cite{POUPON2004233}, astrophysics \cite{refId0} and communication networks \cite{MR2150719}. For more details on Poisson-Voronoi tessellations, i.e., Voronoi tessellations generated by an underlying Poisson  process, we refer the reader to e.g.\  \cite{MR2654678,MR1295245,MR2455326}. We denote by $\mathbf{B}(x,r)$ the open ball centered at $x\in \R^d$ with radius $r>0$. The circumscribed radius of the Voronoi cell $N(x,\nu)$ is defined as
\begin{align*}
	C(x,\nu)=\mathrm{inf}\left\{R\geq0\,:\,\mathbf{B}(x,R)\supset N(x,\nu)\right\},
\end{align*}
i.e., the circumscribed radius is the smallest radius for which the ball centered at the nucleus contains the cell.

Throughout this subsection we consider the stationary Poisson-Voronoi tessellation generated by a Poisson process $\eta_t$ on $\R^d$ with intensity measure $t\lambda_d$, $t > 0$, where $\lambda_d$ is the $d$-dimensional Lebesgue measure. Let $W\subset\mathbb{R}^d$ be a measurable set with $\lambda_d(W)=1$. For any Voronoi cell $N(x,\eta_t)$ with $x\in\eta_t\cap W,$ we take the circumscribed radius of the cell, and we define the point process  $\xi_t$ on the positive half line as 
\begin{align}\label{xi-circ-rad}
	\xi_t
	=\sum_{x\in\eta_t \cap W}\delta_{\alpha_2k_d t^{(d+2)/(d+1)}C(x,\eta_t)^d}.
\end{align}
Here $k_d$ denotes the volume of the $d$-dimensional unit ball, and the constant $\alpha_2>0$ is given by
\begin{equation}\label{alpha2}
\alpha_2
=\left(\frac{2^{d(d+1)}}{(d+1)!}p_{d+1}\right)^{1/(d+1)}
\end{equation}
with
\begin{equation}\label{pd+1}
p_{d+1} := \mathbb{P}\Big( N\Big(0,\sum_{j=1}^{d+1} \delta_{Y_j}\Big) \subseteq  \mathbf{B}(0,  1)\Big),
\end{equation}
where $Y_1, \dots , Y_{d+1}$ are independent and uniformly distributed random points in $\mathbf{B}(0, 2)$. We denote by $T_t$ the first arrival time of $\xi_t$, i.e.,
\begin{align}\label{eq:MinCirc}
	T_t 
	= \min_{x\in\eta_t \cap W} \alpha_2k_d t^{(d+2)/(d+1)}C(x,\eta_t)^d,
\end{align}
which is - up to a rescaling - the $d$-th power of the  minimal circumscribed radius of the cells with nucleus in $W$. Recall that a random variable $Y$ has a Weibull distribution if its cumulative distribution function is given by 
$
\mathbb{P}(Y \leq u)= 1 - e^{-(u/s)^k}
$
for $u\geq0$, and $0$ otherwise; $k>0$ is the shape parameter and $s>0$ is the scale parameter.

\begin{theorem}\label{last-thm}
Suppose $t\geq 1$. Let $\xi_t$ and $T_t$ be the point process and the random variable  given by \eqref{xi-circ-rad} and \eqref{eq:MinCirc}, respectively. Let $Y$ be a Weibull distributed random variable with shape parameter $d+1$ and scale parameter $1$. Then, there exist constants $C_{\mathrm{TV}},C_K>0$ only depending on $d$ such that
	\begin{equation}\label{lst_thm-pr-1}\begin{split}
		 	d_{TV}\big(\xi_t([0,u]),P_{u^{d+1}}\big) \leq  C_{\mathrm{TV}}  \frac{u^{d+2}}{t^{1/(d+1)}}
	\end{split}\end{equation}
	for  $u>0$, and
\begin{align}\label{òlstingqw}
		d_K(T_t , Y)
		& \leq  \frac{C_K}{t^{1/(d+1)}}.
	\end{align}
\end{theorem}
Note that explicit formulas for the constants  $C_{\mathrm{TV}}$ and $C_K$ are given in the proof  of Theorem \ref{last-thm}. In \cite[Theorem 1, Equation (2d)]{MR3252817}, the weak convergence of $T_t$ to $Y$ as $t\to \infty$ is shown. For an underlying inhomogeneous Poisson process, the weak convergence of $\xi_t$ to a Poisson process and the weak convergence of $T_t$ to $Y$ are proven in \cite[Section 3.3]{pianoforte2021criteria}. Although we only consider stationary Poisson processes, we believe that the arguments employed in this subsection may also establish similar results on the minimal circumscribed radius  for more general Poisson processes with a different rate of convergence in $t$ under some constraints on the density (e.g.\ H\"older continuity).  To the best of our knowledge, the present paper is the first time the rates of convergence for the Poisson approximation of $\xi_t([0,u])$ and the Weibull approximation of $T_t$ have been addressed.

The proof of Theorem \ref{last-thm} requires several preparations. 
We set $$s_t = \alpha_2k_d t^{(d+2)/(d+1)}.$$  Let $M_t$ denote the intensity measure of $\xi_t$,  and let the quantities  $\widehat{M}_t$ and $\theta_t$ on $[0,\infty)$ be defined by 
\begin{align*}
	\widehat{M}_t([0,u]) 
	& = t \int_{W}\,\mathbb{E}\big[\mathbf{1}\big\{s_t C(x,\eta_{t} +\delta_x)^d \leq u \big\} \mathbf{1}\big\{\eta_t\big(\mathbf{B}\big(x,4(u/s_t)^{1/d}\big)\big)=d+1\big\} \big] d x, \\
	\theta_t ([0,u]) 
	& =  t \int_{W}\,\mathbb{E}\big[\mathbf{1}\big\{ s_t C(x,\eta_{t} +\delta_x)^d \leq u \big\} \mathbf{1}\big\{\eta_t\big(\mathbf{B}\big(x,4(u/s_t)^{1/d}\big)\big)>d+1\big\} \big] d x
\end{align*}
for $u\geq 0$. For $x\in W$ and $u\geq 0$ we have
\begin{equation}\label{ConditionBall}
\eta_t(\mathbf{B}(x,2(u/s_t)^{1/d}))\geq d+1 \quad \text{whenever} \quad s_t C(x,\eta_t+\delta_x)^d \leq u. 
\end{equation}
This is the case since $s_t C(x,\eta_t+\delta_x)^d \leq u$ implies that the nuclei of the neighboring cells of $x$ are in $\mathbf{B}(x,2(u/s_t)^{1/d})$ and each Voronoi cell has at least $d+1$ neighboring cells. From the Mecke formula and \eqref{ConditionBall} it follows that
$$
M_t([0,u])= \widehat{M}_t([0,u]) + \theta_t ([0,u]), \quad u\geq 0.
$$

\begin{lemma}\label{calc-meas}
For all $u>0$ and $t>0$,
$$
\widehat{M}_t([0,u])  
=  u^{d+1} \exp\Big(-\frac{4^d u}{\alpha_2 t^{1/(d+1)}}\Big),
\   
\theta_t ([0,u]) \leq 	\frac{2^{d(d+3)} }{\alpha_2 p_{d+1}} \frac{u^{d+2}}{t^{1/(d+1)}} \  \text{and} \ M_t([0,u])\leq \frac{u^{d+1}}{p_{d+1}}. 
$$
\end{lemma}

\begin{proof}
First we compute $\widehat{M}_t([0,u]) $. 
From \eqref{ConditionBall} and the definition of $p_{d+1}$ in \eqref{pd+1} we derive
\begin{align*}
		\widehat{M}_t([0,u]) 
		& = t \int_{W} e^{-2^d k_d tu/s_t}\frac{\big(2^dk_dtu/s_t\big)^{d+1}}{(d+1)!}p_{d+1} \\
		& \qquad \quad \times \mathbb{P}\big(\eta_t\big(\mathbf{B}\big(x,4(u/s_t)^{1/d}\big)\setminus\mathbf{B}\big(x,2(u/s_t)^{1/d}\big)\big)=0\big)d x.
\end{align*}
Substituting $s_t = \alpha_2k_d t^{(d+2)/(d+1)}$ and  $\alpha_2= \big(\frac{2^{d(d+1)}}{(d+1)!}p_{d+1}\big)^{1/(d+1)}$  into the previous equation implies that the right-hand side equals
\begin{align*}
	&  u^{d+1} \int_{W} \exp\bigg(-\frac{2^d u}{\alpha_2 t^{1/(d+1)}} -t\lambda_d\big(\mathbf{B}	\big(x,4(u/s_t)^{1/d}\big)\setminus\mathbf{B}\big(x,2(u/s_t)^{1/d}\big)\big)\bigg)
	d x
	\\
	& =  u^{d+1} \exp\bigg(-\frac{2^d u}{\alpha_2 t^{1/(d+1)}}  -\frac{2^d u}{\alpha_2 t^{1/(d+1)}}(2^d-1)\bigg) 
	= u^{d+1} \exp\Big(-\frac{4^d u}{\alpha_2 t^{1/(d+1)}}\Big),
\end{align*}
which completes the first part of the proof.
	
For $u>0$, we have
$$
\theta_t ([0,u]) 
 \leq t \int_{W}\,\mathbb{E}\big[\mathbf{1}\big\{\eta_t\big(\mathbf{B}\big(x,4(u/s_t)^{1/d}\big)\big)>d+1\big\} \big] d x = t  \sum_{k=d+2}^\infty e^{-\beta_t}\frac{\beta_t^k}{k!}
$$
with $\beta_t =4^d k_d t u/s_t$. Elementary calculations imply that
\begin{align*}
	t  \sum_{k=d+2}^\infty e^{-\beta_t}\frac{\beta_t^k}{k!}
	& = t  \beta_t^{d+2} \sum_{k=d+2}^\infty{e^{-\beta_t}\frac{\beta_t^{k-d-2}}{k!}}  
 = t  \beta_t^{d+2} \sum_{\ell=0}^\infty{e^{-\beta_t}\frac{\beta_t^{\ell}}{(\ell+d+2)!}} 	\\
	& 
	\leq \frac{t  \beta_t^{d+2}}{(d+2)!} = \frac{t \big(4^d k_d t u/s_t\big)^{d+2}}{(d+2)!} .
\end{align*}
Substituting $s_t = \alpha_2k_d t^{(d+2)/(d+1)}$ and $\alpha_2= \big(\frac{2^{d(d+1)}}{(d+1)!}p_{d+1}\big)^{1/(d+1)}$ into the latter term yields
$$
\theta_t([0,u])\leq 
\frac{2^{d(d+3)} }{\alpha_2 p_{d+1}} \frac{u^{d+2}}{t^{1/(d+1)}}  ,
$$
which is the desired result.
	
From the Mecke formula, \eqref{ConditionBall} and the same arguments as above, we obtain
\begin{align*}
M_t([0,u]) & \leq t \int_W \mathbb{P}(\eta_t(\mathbf{B}(x,2(u/s_t)^{1/d}))\geq d+1) d x = t \sum_{k=d+1}^\infty \frac{(2^d k_d t u/s_t)^k}{k!} e^{-2^d k_d t u/s_t} \\
& \leq \frac{t (2^d k_d t u/s_t)^{d+1}}{(d+1)!} = \frac{2^{d(d+1)} k_d^{d+1} t^{d+2} u^{d+1} }{k_d^{d+1} \frac{2^{d(d+1)} p_{d+1}}{(d+1)!} (d+1)! t^{d+2} } = \frac{u^{d+1}}{p_{d+1}},
\end{align*}
which concludes the proof.
\end{proof}

We now provide a statement from \cite[Lemma 3.14]{pianoforte2021criteria}, which will be employed in the proof of the subsequent proposition.

\begin{lemma}\label{lemmconvhull}
	Let $x_0,\hdots,x_{d+1}\in\mathbb{R}^d$ be in general position (i.e., no $k$-dimensional affine subspace of $\mathbb{R}^d$ with $k\in\{0,\hdots,d-1\}$ contains more than $k+1$ of the points) and assume that $N(x_0,\sum_{j=0}^{d+1} \delta_{x_i})$ is bounded. Then
$N(x_i,\sum_{j=0}^{d+1} \delta_{x_i})$ is unbounded for any $i\in\{1,\hdots,d+1\}$.
\end{lemma}
Next we construct a random variable that satisfies \eqref{Mek2} for $\xi_t([0,u])$ with remainder terms $q_{i},i\in\N_0$, which vanish as $t\to\infty$. By $B^c$ we denote the complement of $B\subset\R^d$ and by $\eta_t|_B$ the restriction of $\eta_t$ to $B$.

\begin{proposition}\label{prp-circ-Z}
	Let $X$ be uniformly distributed in $W$ and independent of $\eta_t$. Then for $u> 0$,
	$$
	k\mathbb{P}(\xi_t([0,u])=k) =   \widehat{M_t}([0,u]) \mathbb{P}( \xi_t([0,u]) + Z_{t,u} = k-1 ) + q_{k-1} (t,u), \quad k\in\N,
	$$
	with
	$$
	Z_{t, u} =\xi_t\big(\eta_t|_{\mathbf{B}(X,4(u/s_t)^{1/d})^c}\big) ([0,u]) -  \xi_t([0,u])
	$$
	and
	\begin{align*}
		q_i(t,u)
		&= t \int_{W}\mathbb{E}\Big[\mathbf{1}\big\{s_t C(x,\eta_{t} + \delta_x)^d \leq u \big\} \mathbf{1}\big\{\eta_t\big(\mathbf{B}\big(x,4(u/s_t)^{1/d}\big)\big)>d+1\big\}
		\\
		&\qquad \qquad \times\mathbf{1}\Big\{\sum_{y\in\eta_{t}\cap W} \mathbf{1}\big\{s_t C(y,\eta_{t}+\delta_x)^d\leq u\big\}=i\Big\} \Big] d x
	\end{align*}
	for $i\in\N_0$.
\end{proposition}

\begin{proof}
	The Mecke equation implies for $k\in\N$ that
	\begin{align*}
		& k\mathbb{P}(\xi_t([0,u])=k)
		= t \int_{W}\mathbb{E}\big[\mathbf{1}\big\{ s_t C(x,\eta_{t}+\delta_x)^d \leq u \big\}\mathbf{1}\big\{\xi_t(\eta_t + \delta_x)([0,u])= k\big\} \big] d x
		\\
		& = t \int_{W}\mathbb{E}\Big[\mathbf{1}\big\{s_tC(x,\eta_{t}+\delta_x)^d \leq u \big\}\mathbf{1}\Big\{ \sum_{y\in\eta_{t}\cap W} \mathbf{1}\big\{s_tC(y,\eta_{t}+\delta_x)^d \leq u\big\}=k-1\Big\} \Big] d x.
	\end{align*}
	Now we divide the integral in
	\begin{align*}
		A_k + q_{k-1}(t,u) 
		: = & \,\, t \int_{W} \mathbb{E}\Big[\mathbf{1}\big\{ s_t C(x,\eta_{t} +\delta_x)^d \leq u \big\} \mathbf{1}\big\{\eta_t\big(\mathbf{B}\big(x,4(u/s_t)^{1/d}\big)\big)=d+1\big\}
		\\
		& \qquad\qquad \times \mathbf{1}\Big\{\sum_{y\in\eta_{t}\cap W} \mathbf{1}\big\{s_tC(y,\eta_{t}+\delta_x)^d \leq u\big\}=k-1\Big\} \Big] d  x
		\\
		&+ t \int_{W}\mathbb{E}\Big[\mathbf{1}\big\{s_t C(x,\eta_{t} + \delta_x)^d \leq u \big\} \mathbf{1}\big\{\eta_t\big(\mathbf{B}\big(x,4(u/s_t)^{1/d}\big)\big)>d+1\big\}
		\\
		& \qquad\qquad \times \mathbf{1}\Big\{\sum_{y\in\eta_{t}\cap W} \mathbf{1}\big\{s_t C(y,\eta_{t}+\delta_x)^d \leq u\big\} =k-1\Big\} \Big] d x.
	\end{align*}
	Then, it is enough to show that $A_k=\widehat{M_t}([0,u]) \mathbb{P}( \xi_t([0,u]) + Z_{t,u} = k-1 )$. 
In order to simplify the notation throughout this proof, we write
$$
\mathbf{B}_2(x):=\mathbf{B}\big(x,2(u/s_t)^{1/d}\big) \quad \text{and} \quad  \mathbf{B}_4(x):=\mathbf{B}\big(x,4(u/s_t)^{1/d}\big),\quad  x\in\R^d.
$$
In case there are only $d+1$ points of $\eta_t$ in $\mathbf{B}_4(x)$, we have by \eqref{ConditionBall} that $s_t C(x, \eta_t + \delta_x)^d \leq u$ only if the $d+1$ elements of $\eta_t$ belong to $\mathbf{B}_2(x)$. Therefore we obtain
\begin{equation}\begin{split}\label{A:fisrtident}
A_k 
& =  t \int_{W}\,\mathbb{E}\Big[\mathbf{1}\big\{ s_t C(x,\eta_{t} +\delta_x)^d \leq u \big\}  
\mathbf{1}\big\{\eta_t(\mathbf{B}_4(x)\setminus \mathbf{B}_2(x))=0,\eta_t(\mathbf{B}_2(x))=d+1\big\} \\
& \qquad \qquad \quad\times \mathbf{1}\Big\{\sum_{y\in\eta_{t}\cap W} \mathbf{1}\{s_t C(y,\eta_{t}+\delta_x)^d \leq u \big\}=k-1\Big\} \Big] d x .
\end{split}\end{equation}
The observation that
\begin{align}\label{obs: circRad}
s_t C(y, \eta_t + \delta_x)^d \leq u
\quad \text{if and only if} \quad
s_t C\big(y, (\eta_t + \delta_x)|_{\mathbf{B}_2(y)}\big)^d \leq u
\end{align}
for $y\in\eta_t$ establishes that 
\begin{align*}
A_k 
& =  t \int_{W}\,\mathbb{E}\Big[\mathbf{1}\big\{ s_t C(x,\eta_{t} +\delta_x)^d \leq u \big\} 
\mathbf{1}\big\{\eta_t(\mathbf{B}_4(x)\setminus \mathbf{B}_2(x))=0,\eta_t(\mathbf{B}_2(x))=d+1\big\}
\\
	& \qquad \times \mathbf{1}\Big\{\xi_t(\eta_t|_{\mathbf{B}_4(x)^c} )([0,u]) 
+ \sum_{y\in\eta_{t}\cap \mathbf{B}_2(x) \cap W} \mathbf{1}\big\{s_t C(y,\eta_{t}+\delta_x)^d \leq u \big\} =k-1\big\} \Big] d x.
\end{align*}
Suppose that  $s_t C(x,\eta_{t} +\delta_x)^d \leq u$ and that there are exactly $d+1$ points $y_1,\dots , y_{d+1}$ of $\eta_t$ in $\mathbf{B}_2(x)$ and $\eta_t\cap \mathbf{B}_4(x)\cap \mathbf{B}_2(x)^{c}=\emptyset$. From Lemma \ref{lemmconvhull} it follows that the Voronoi cells $N(y_i, \eta_t|_{\mathbf{B}_4(x)} +\delta_x), i=1,\dots,d+1$, are unbounded. In particular, we have
	$$
	C(y_i, \eta_t + \delta_x)> (u/s_t)^{1/d}, \quad i=1,\dots, d+1.
	$$
Together with the same arguments used to show \eqref{A:fisrtident} and independence, this implies that
	\begin{align*}
	A_k 
	& =  t \int_{W}\,\mathbb{E}\big[\mathbf{1}\big\{ s_t C(x,\eta_{t} +\delta_x)^d \leq u \big\} 
	\mathbf{1}\{\eta_t(\mathbf{B}_4(x)\setminus \mathbf{B}_2(x))=0,\eta_t(\mathbf{B}_2(x))=d+1\}
	\\
	& \qquad \qquad \times \mathbf{1}\{\xi_t(\eta_t|_{\mathbf{B}_4(x)^c} )([0,u]) =k-1
 \}\big] d x
 \\
& = t \int_{W}\,\mathbb{E}\big[\mathbf{1}\big\{s_t C(x,\eta_{t} +\delta_x)^d \leq u \big\} \mathbf{1}\{\eta_t(\mathbf{B}_4(x))=d+1\}\big] 
\\
& \qquad \quad \times
\mathbb{P}\big( \xi_t(\eta_t|_{\mathbf{B}_4(x)^c} \big) ([0,u]) = k-1\big) d x .
\end{align*}
Then, because the expectation in the latter equation does not depend on the choice of $x\in W$, we have that 
\begin{align*}
A_k & =\widehat{M_t}([0,u])\int_W \mathbb{P}\big( \xi_t\big(\eta_t|_{\mathbf{B}_4(x)^c}\big) ([0,u]) = k-1\big) dx \\
& = \widehat{M_t}([0,u]) \mathbb{P}( \xi_t([0,u]) + Z_{t,u}= k-1 )
\end{align*}
with
	$$
	Z_{ t, u} =  \xi_t\big(\eta_t|_{\mathbf{B}_4(X)^c}\big) ([0,u]) - \xi_t([0,u]) .
	$$
This and $\mathbf{B}_4(X)= \mathbf{B}\big(X,4(u/s_t)^{1/d}\big)$ give the desired conclusion.
\end{proof}

\begin{lemma}\label{lem:BoundZtu}
For $u>0$, $t>0$ and $Z_{t,u}$ as in Proposition \ref{prp-circ-Z},
$$
\mathbb{E}[|Z_{t,u}|] \leq \frac{6^d}{\alpha_2 p_{d+1}} \frac{u^{d+2}}{t^{(d+2)/(d+1)}}.
$$
\end{lemma}

\begin{proof}
For $x\in W$ it follows from the observation in \eqref{obs: circRad} that
$$
0 \leq \xi_t([0,u]) - \xi_t\big(\eta_t|_{\mathbf{B}(x,4(u/s_t)^{1/d})^c}\big) ([0,u])  \leq \sum_{y\in\eta_t\cap W \cap \mathbf{B}(x,6(u/s_t)^{1/d})} \mathbf{1}\{ s_t C(y,\eta_t)^d \leq u \}.
$$
By the Mecke formula and the stationarity of $\eta_t$, we obtain
\begin{align*}
\mathbb{E} \sum_{y\in\eta_t\cap W \cap \mathbf{B}(x,6(u/s_t)^{1/d})} \!\!\!\!\!\! \mathbf{1}\{ s_t C(y,\eta_t)^d \leq u \} & \leq t \lambda_{d}(W \cap \mathbf{B}(x,6(u/s_t)^{1/d})) \mathbb{P}( s_t C(0,\eta_t+\delta_0)^d \leq u ) \\
& \leq \frac{6^d u}{\alpha_2 t^{(d+2)/(d+1)}} t\mathbb{P}( s_t C(0,\eta_t+\delta_0)^d\leq u ).
\end{align*}
From Lemma \ref{calc-meas} we deduce
$$
t \mathbb{P}( s_t C(0,\eta_t+\delta_0)^d\leq u ) = M_t([0,u]) \leq \frac{u^{d+1}}{p_{d+1}},
$$
which proves the assertion.
\end{proof}

\begin{lemma}\label{bnd-xi-crc-0}
For $u>0$ and $t>0$, 
	\begin{align*}
\mathbb{P}(T_t>u)	= \mathbb{P}(\xi_t([0,u])=0) \leq e^{-\widehat{M}_t([0,u])}.
	\end{align*}
\end{lemma}
\begin{proof}
The first identity is obvious. Let  $Z_{t,u}$ be the random variable defined in Proposition \ref{prp-circ-Z}. Since $Z_{t,u}\leq 0$, $\mathbb{P}(\xi_t([0,u])+ Z_{t,u}\geq 0)=1$ and $q_i(t,u)\geq 0$ for all $i\in\N_0$, the inequality follows from Proposition \ref{prop-bound-0} b).
\end{proof}

In the next lemma, we combine the results obtained above and  Theorem \ref{Thmmm2} to derive intermediate bounds on the quantities considered in Theorem \ref{last-thm}.
\begin{lemma}\label{lem:IntermediateBounds}
	For  $u>0$ and $t>0$,
	\begin{align}\label{alf}
		d_{TV}\big(\xi_t([0,u]),P_{\widehat{M}_t([0,u])}\big)  
		& \leq \frac{6^d}{\alpha_2 p_{d+1}} \frac{u^{d+2}}{t^{(d+2)/(d+1)}}  + \theta_t([0,u])
	\end{align}
	and
	\begin{equation}\label{alpw}\begin{split}
			0 
			& \leq  e^{-\widehat{M}_t([0,u])} -  \mathbb{P}(T_t>u) 
	\leq \bigg( 1+\frac{1}{\widehat{M}_t([0,u])} \bigg) \frac{6^d}{\alpha_2 p_{d+1}} \frac{u^{d+2}}{t^{(d+2)/(d+1)}} + \frac{2  \theta_t([0,u])}{\widehat{M}_t([0,u])^{2}}.
	\end{split}\end{equation}
\end{lemma}

\begin{proof}
From Proposition \ref{prp-circ-Z} it follows that the assumptions of Theorem \ref{Thmmm2} are satisfied. Then, \eqref{fsxx2} in Theorem \ref{Thmmm2} yields
\begin{align*}
& d_{TV}(\xi_t([0,u]),P_{\widehat{M}_t([0,u])}) 
\leq (1 \wedge \widehat{M}_t([0,u]) ) \mathbb{E}[ \vert Z_{t,u} \vert ] + \big(1 \wedge \widehat{M}_t([0,u])^{-1/2}\big) \theta_t([0,u])
\end{align*}
so that \eqref{alf} follows from Lemma \ref{lem:BoundZtu}.
	
Let us now prove \eqref{alpw}. 	From Lemma \ref{bnd-xi-crc-0}, \eqref{rrrf2} in Theorem \ref{Thmmm2} with $m=1$ and $\sum_{i=1}^\infty q_i(t,u)\leq \theta_t([0,u])$ we obtain
\begin{align*}
0 & \leq e^{-\widehat{M}_t([0,u])} -  \mathbb{P}(T_t>u) 
			\leq \frac{\mathbb{E}\big[ \vert Z_{ t,u} \vert \big]}{\widehat{M}_t([0,u])} +  \mathbb{E}\big[ \vert Z_{ t,u} \vert  \big] 
		+ \frac{q_{0}(t,u)}{\widehat{M}_t([0,u])} +  \frac{\theta_t([0,u])}{\widehat{M}_t([0,u])^{2}}.
\end{align*}
The first two terms on the right-hand side can be bounded by Lemma \ref{lem:BoundZtu}.
Recall that
\begin{align*}
q_{0}(t,u)
& = t \int_{W}\mathbb{E}\Big[\mathbf{1}\big\{ s_t C(x,\eta_{t}+\delta_x)^d \leq u\big\} \mathbf{1}\big\{\eta_t\big(\mathbf{B}\big(x,4(u/s_t)^{1/d}\big)\big)>d+1\big\} \\
& \qquad \qquad \times \mathbf{1}\Big\{\sum_{y\in\eta_{t}\cap W} \mathbf{1}\{s_t C(y,\eta_{t}+\delta_x)^d \leq u \} =0\Big\} \Big] dx.
\end{align*}
Since the product of the first two indicator functions is increasing with respect to additional points, while the third indicator function is decreasing, it follows from \cite[Theorem 20.4]{MR3791470} that
\begin{align*}
q_{0}(t,u)
& \leq t \int_{W}\mathbb{E}\Big[\mathbf{1}\big\{ s_t C(x,\eta_{t}+\delta_x)^d \leq u\big\} \mathbf{1}\big\{\eta_t\big(\mathbf{B}\big(x,4(u/s_t)^{1/d}\big)\big)>d+1\big\}\Big] \\
& \qquad \qquad \times \mathbb{P}\Big(\sum_{y\in\eta_{t}\cap W} \mathbf{1}\{s_t C(y,\eta_{t}+\delta_x)^d \leq u \}=0\Big)  dx.
\end{align*}
Now Lemma \ref{bnd-xi-crc-0} and the elementary inequality $v e^{-v}\leq 1$ for $v\geq 0$ lead to
$$
q_{0}(t,u) \leq \theta_t([0,u]) \mathbb{P}(\xi_t([0,u])=0) \leq \theta_t([0,u]) e^{-\widehat{M}_t([0,u])} \leq \frac{\theta_t([0,u])}{\widehat{M}_t([0,u])},
$$
which concludes the proof.
\end{proof}

\begin{proof}[Proof of Theorem \ref{last-thm}]
	Let $u>  0$ be fixed. From \eqref{alf} in Lemma \ref{lem:IntermediateBounds}, Lemma \ref{calc-meas} and $t\geq 1$ it follows that
	\begin{align}\label{BoundTV3.5}
		d_{TV}\big(\xi_t([0,u]),P_{\widehat{M}_t([0,u])}\big)  
		&\leq \frac{6^d }{\alpha_2 p_{d+1}} \frac{u^{d+2}}{t^{(d+2)/(d+1)}}  + \theta_t([0,u])
		\leq \frac{6^d +  2^{d(d+3)} }{\alpha_2  p_{d+1}} \frac{u^{d+2}}{t^{1/(d+1)}} .
	\end{align}
Using a well-known bound for the total variation distance between two Poisson distributed random variables, Lemma \ref{calc-meas} and the inequality $1-e^{-v}\leq v$ for $v\geq 0$, we obtain
$$
d_{TV}\big(P_{u^{d+1}}, P_{\widehat{M}_t([0,u])}\big)
\leq u^{d+1} -\widehat{M}_t([0,u]) = u^{d+1} \bigg( 1 - \exp\left( - \frac{4^d u}{\alpha_2 t^{1/(d+1)}} \right) \bigg)
\leq \frac{4^d u^{d+2}}{\alpha_2 t^{1/(d+1)}}.
$$
Now the triangle inequality yields 
	$$
	d_{TV}\big(\xi_t([0,u]),P_{u^{d+1}}\big) \leq  \frac{3 \cdot 2^{d(d+3)} }{ \alpha_2 p_{d+1}} \frac{u^{d+2}}{t^{1/(d+1)}},
	$$
	which proves \eqref{lst_thm-pr-1}.
	
	Let us now show \eqref{òlstingqw}. From \eqref{BoundTV3.5} and Lemma \ref{bnd-xi-crc-0} we have that, for $u\in[0,1]$,
	\begin{align*}
		0 
		\leq e^{-\widehat{M}_t([0,u])} - \mathbb{P}(T_t >u) 
		\leq \frac{ 2^{d(d+3)+1} }{\alpha_2 p_{d+1}} \frac{1}{t^{1/(d+1)}}.
	\end{align*}
In the following we consider the case $1 \leq u \leq t^{1/(2d+2)} \tau$ with $\tau =\alpha_2/4^d$. From Lemma \ref{calc-meas} and $t\geq 1$ we obtain
\begin{equation}\label{eq:BoundMTilde}
u^{d+1} \geq \widehat{M}_t([0,u]) \geq \frac{u^{d+1}}{e}.
\end{equation}
Together with Lemma \ref{bnd-xi-crc-0}, \eqref{alpw} in Lemma \ref{lem:IntermediateBounds}, Lemma \ref{calc-meas} and $u\geq 1$ we obtain
\begin{align*}
0  \leq e^{-\widehat{M}_t([0,u])} -  \mathbb{P}(T_t>u) 
&  \leq \bigg( 1+\frac{1}{\widehat{M}_t([0,u])} \bigg) \frac{6^d}{\alpha_2 p_{d+1}} \frac{u^{d+2}}{t^{(d+2)/(d+1)}} + \frac{2  \theta_t([0,u])}{\widehat{M}_t([0,u])^{2}} \\
& \leq (1 + e) \frac{6^d}{\alpha_2 p_{d+1}} \frac{u^{d+2}}{t^{(d+2)/(d+1)}} + 2 e^2 \frac{1}{u^{2d+2}} \frac{2^{d(d+3)} }{\alpha_2 p_{d+1}} \frac{u^{d+2}}{t^{1/(d+1)}}.
\end{align*}
Using $	1\leq u^{d+2}\leq t^{(d+2)/(2d+2)} \alpha_2^{d+2}/ 4^{d(d+2)}$, $t\geq 1$ and the definition of $\alpha_2$ in \eqref{alpha2}, we deduce 
\begin{align*}
0  \leq e^{-\widehat{M}_t([0,u])} -  \mathbb{P}(T_t>u) & \leq \frac{(1+e) 6^d}{4^{d(d+2)}} \frac{\alpha_2^{d+1}}{p_{d+1}} \frac{1}{t^{1/(d+1)}} + 2 e^2 \frac{1}{u^{2d+2}} \frac{2^{d(d+3)} }{\alpha_2 p_{d+1}} \frac{u^{d+2}}{t^{1/(d+1)}} \\
& \leq  \frac{1}{t^{1/(d+1)}} + \frac{2^{d(d+3)+4} }{\alpha_2 p_{d+1}} \frac{1}{t^{1/(d+1)}}
\end{align*}
so that
$$
\sup_{u\in [0,t^{1/(2d+2)}\tau]} \vert e^{-\widehat{M}_t([0,u])} -  \mathbb{P}(T_t>u) \vert \leq \bigg[1 + \frac{2^{d(d+3)+4} }{\alpha_2 p_{d+1}}\bigg] \frac{1}{t^{1/(d+1)}}.
$$
Moreover, by Lemma \ref{calc-meas}, \eqref{eq:BoundMTilde} and elementary arguments we obtain  for $0\leq u \leq t^{1/(2d+2)}\tau$ that
\begin{equation*}\begin{split}
0	& \leq  e^{-\widehat{M}_t([0,u])} - e^{-u^{d+1}} \leq   \left[ u^{d+1} - \widehat{M}_t([0,u]) \right] e^{-\widehat{M}_t([0,u])} \\
	& \leq \frac{ 4^d u^{d+2}}{\alpha_2 t^{1/(d+1)}}e^{- u^{d+1}e^{-1}}	\leq \frac{4^d e^{\frac{d+2}{d+1}} }{\alpha_2 t^{1/(d+1)}}  \leq  \frac{2^{2d+3} }{\alpha_2 t^{1/(d+1)}},
\end{split}\end{equation*}
where we used the inequalities $1-e^{-x}\leq x$ and $e^{-x^{d+1}}x^{d+2} \leq 1$ for $x\geq0$. This implies that
$$
\sup_{u\in [0,t^{1/(2d+2)}\tau]} \vert e^{-u^{d+1}} -  \mathbb{P}(T_t>u) \vert \leq \bigg[1 + \frac{2^{d(d+3)+4} }{\alpha_2 p_{d+1}}+ \frac{2^{2d+3} }{\alpha_2 }\bigg] \frac{1}{t^{1/(d+1)}}.
$$
On the other hand, $x^2 e^{-x^{d+1}}\leq 1$ for $x\geq 0$ leads to
\begin{align*}
\exp\left( -(t^{1/(2d+2)}\tau)^{d+1} \right)
 \leq  \frac{1}{( t^{1/(2d+2)}\tau)^2} \leq \frac{16^d}{\alpha_2^2} \frac{1}{t^{1/(d+1)}}.
\end{align*}
Combining the two previous inequalities gives a bound for $\mathbb{P}\left(T_t>t^{1/(2d+2)}\tau\right)$ and it implies
\begin{align*}
& \sup_{u\in [0,\infty)} \vert e^{-u^{d+1}} -  \mathbb{P}(T_t>u) \vert
\\
& \leq \max \big\{\sup_{u\in [0,t^{1/(2d+2)}\tau]} \vert e^{-u^{d+1}} -  \mathbb{P}(T_t>u) \vert, \mathbb{P}\left(T_t>t^{1/(2d+2)}\tau\right) , \exp\left( -(t^{1/(2d+2)}\tau)^{d+1} \right)\big\}
\\
& \leq \bigg[ 1 + \frac{2^{d(d+3)+4} }{\alpha_2 p_{d+1}}+ \frac{2^{2d+3} }{\alpha_2 }+  \frac{16^{d}}{\alpha_2^2}\bigg] \frac{1}{t^{1/(d+1)}}.
\end{align*}
Now the identity $\mathbb{P}(T_t>0)= 1$ concludes the proof.
\end{proof}

\subsection{Maximal inradii of Poisson-Voronoi tessellations}
\label{third-appl-proof}
In this subsection, we consider  the  inradii of stationary Poisson-Voronoi tessellations.  Recall that the inradius of a cell is the largest radius for which the ball centered at the nucleus is contained in the cell.     The aim is to continue the work started in   \cite{MR3252817} by proving that the Kolmogorov distance between a transform of the largest inradius  and a Gumbel  random variable converges to $0$ at a rate of $\log(t)/\sqrt{t}$ as the intensity $t$ of the underlying Poisson process goes to infinity. More details on Poisson-Voronoi tessellations are given in Subsection \ref{circ-rad-proof}.  

Let $W\subset\mathbb{R}^d$ be a measurable set with Lebesgue measure $\lambda_d(W)=1$.  Let $\eta_{t}$ be a Poisson process on $\R^d$ with intensity measure $t\lambda_d, t>0.$ 
Consider the measurable function $h_t: W \times \mathcal{N}(\R^d) \rightarrow \mathbb{R}$ defined as
\begin{align*}
	h_t(x, \mu) 
	= \underset{y\in \mu\setminus\{x\}}{\min}\,t k_d\Vert x-y\Vert^d - \log(t),
\end{align*}
where $k_d$ is the volume of the $d$-dimensional unit ball. Note that for any $x\in\eta_t$,  ${\mathrm{min}}\{\Vert x-y\Vert \, : \,  y\in\eta_t\setminus \{x\}\}$ is twice the    inradius of the Voronoi cell with nucleus $x$ generated by $\eta_t$. Then, the random variable
\begin{align}\label{eq:maxInr}
	T_t = \underset{x\in \eta_t \cap W }{\max} \, h_t(x ,\eta_t)  
\end{align}
is a transform of the maximal inradius over the cells with nucleus in $W$.  We define the point process  $\xi_t$  as
\begin{equation}\label{max-irr-xi}\begin{split}
		\xi_t 
		& = \xi_t(\eta_t) =\sum_{x\in\eta_t\cap W} \delta_{h_t(x,\eta_t)}.
\end{split}\end{equation}
Recall that a random variable $Y$ has a standard Gumbel distribution if its cumulative distribution function is given by
$
\mathbb{P}(Y \leq u)= e^{-e^{-u}}
$
for $u\in\R$.
\begin{theorem}\label{Thm...}
	Suppose $t>e^2$. Let $T_t$  and $\xi_t$ be the random variable and the point process   given by \eqref{eq:maxInr} and \eqref{max-irr-xi},  respectively. Let $Y$ be a random variable with a standard Gumbel distribution. Then,
	\begin{align}\label{ineq-th-1.3.3}
		d_{TV}\big(\xi_t((u,\infty)), P_{e^{-u}}\big)
		\leq 2^{d} \frac{u+\log(t)}{ e^{u/2} \sqrt{t} } +  \frac{u+\log(t)}{e^{u} t}
	\end{align}
	for $u>-\log(t)$, and
	\begin{equation}\begin{split}\label{maxmintest}
			& d_{K}( T_t, Y )
			\leq [2^{d+2}( 4^d +2^d+2)+1]\frac{\log(t)}{\sqrt{t}} .
	\end{split}\end{equation}
\end{theorem}

The main achievement of Theorem \ref{Thm...} is the rate of convergence for the Kolmogorov distance in \eqref{maxmintest}. In \cite[Theorem 1, Equation (2a)]{MR3252817}, the weak convergence of $T_t$ to a Gumbel random variable is proven. For $d=2$ one obtains from the proof of \cite[Proposition 8]{CHENAVIER20142917} that for any fixed $u\in\R$  the  difference between $\mathbb{P}(T_t\leq u)$ and $\mathbb{P}(Y\leq u)$ behaves like $O(\log(t)/\sqrt{t})$, where the constant hidden in the big-O-notation depends on $u$. However this result does not permit to bound the difference between $\mathbb{P}(T_t\leq u)$ and $\mathbb{P}(Y\leq u)$ uniformly in $u\in\R$, whence it does not lead to a bound for the Kolmogorov distance. Note that \cite[Proposition 8]{CHENAVIER20142917} concerns the maximal inradii of planar Gauss-Voronoi tessellations, which are generated by a Poisson cluster process and include planar Poisson-Voronoi tessellations as a special case. For this model it is shown that for any fixed $u\in\R$, $\vert\mathbb{P}(T_t\leq u) - \mathbb{P}(Y\leq u)\vert$ behaves like $O(\log(t)^{-1/2})$, where the big-O-term depends on $u$.

For an underlying inhomogeneous Poisson process, the weak convergence of $\xi_t$ to a Poisson process and the weak convergence of $T_t$ to $Y$ are established in \cite[Section 3.2]{pianoforte2021criteria}, and for an underlying inhomogeneous binomial point process, the weak convergence of  $T_t$ to $Y$ is studied in \cite[Theorem 1]{MR731691}. As for the results stated in Subsection \ref{circ-rad-proof} about the minimal circumscribed radius, we believe that similar arguments as in this subsection could lead to comparable results with a different rate of convergence in $t$ for the maximal inradius of a Voronoi tessellation generated by an inhomogeneous Poisson processes under some constraints on the density.

Counting cells whose inradius is larger than a given value is equivalent to counting isolated vertices in random geometric graphs. The related problem of finding the longest edge of a $k$-nearest neighbor graph or a minimal spanning tree  is studied, for example, in \cite{MR1442317} or \cite[Chapter 8]{MR1986198} for underlying finite Poisson processes or binomial point processes, where one needs to take care of boundary effects.

Since the proof of Theorem \ref{Thm...} is based on Theorem \ref{Thmmm}, together with the second inequality of \eqref{fsxx} in Theorem \ref{Thmmm}, the same arguments used to show \eqref{ineq-th-1.3.3}  may also lead to  a bound on the  Wasserstein distance between $\xi_t((u,\infty))$ and $P_{e^{-u}}$. 

For the proof of Theorem \ref{Thm...} we introduce some notation. By $M_t$ we denote  the intensity measure of $\xi_t$. For $u> -\log(t)$, set
\begin{align}\label{v_t}
	v_t 
	= v_t(u)=\bigg(\frac{ u+\log(t)}{tk_d}\bigg)^{1/d}.
\end{align}
Then, for  $u> -\log(t)$ we have
\begin{align*}
	M_t((u,\infty)) 
	& = t \int_W \mathbb{E}\big[\mathbf{1}\{h_t(x,\eta_t+ \delta_x)> u\}\big] dx =t \int_W \mathbb{P}\big( \eta_t(\textbf{B}(x, v_t))=0\big) dx
	\\
	& = t \int_W e^{-t v_t^d k_d} d x 
	=  t e^{-u-\log(t)} 
	=  e^{-u}.
\end{align*}

Let $X$ be a uniformly distributed random vector in $W$ independent of $\eta_t$.
In the next proposition we show that for each $u > -\log(t)$, and for an opportune choice of a random ball $\mathbf{B}$ centered at $X$, the  random variable $\xi_t(\eta_t|_{\mathbf{B}^c})((u,\infty)) - \xi_t((u,\infty))$  satisfies \eqref{Mek} for $\xi_t((u,\infty))$.
\begin{proposition}\label{prp.Sec-3.3}
	For any $t>e$ and  $u> -\log(t)$,
	$$
	k\mathbb{P}(\xi_t((u,\infty))=k)  
	= M_t((u,\infty))  \mathbb{P}(\xi_t((u,\infty)) + Z_t(u)= k-1), \quad k\in\N,
	$$
	where the random variable $Z_t(u)$ is defined as
	$$
	Z_t(u)
	= \xi_t(\eta_t|_{\textbf{B}(X,v_t)^c})((u,\infty)) - \xi_t((u,\infty))
	$$
	with $v_t=v_t(u)$ given by \eqref{v_t}.
\end{proposition}

\begin{proof}
	Let $B=(u,\infty)$ with $u>-\log(t)$. The Mecke equation yields for $k\in\mathbb{N}$ that
	\begin{align*}
		k\mathbb{P}(\xi_t(B)=k)  
		&= t \int_{W} \mathbb{E}\big[ \mathbf{1}\{h_t(x,\eta_t + \delta_x)> u\}\mathbf{1}\{\xi_t(\eta_t + \delta_x)(B)=k\}\big] d x.
	\end{align*}
	Since $h_t(x,\eta_t + \delta_x)>u$ if and only if $\eta_t(\textbf{B}(x, v_t))=0$, the right-hand side equals
	\begin{align*}
		& t \int_{W} \mathbb{E}\big[ \mathbf{1}\{\eta_t(\textbf{B}(x, v_t))=0\}\mathbf{1}\{\xi_t(\eta_t|_{\textbf{B}(x, v_t)^c})(B)=k-1\}\big] d x
		\\
		& = t \int_{W} \mathbb{P}\big( \eta_t(\textbf{B}(x, v_t))=0\big) \mathbb{E}\big[\mathbf{1}\{\xi_t(\eta_t|_{\textbf{B}(x,v_t)^c})(B)=k-1\}\big] d x
		\\
		& = e^{-u}\int_{W}\mathbb{P}\big(\xi_t(\eta_t|_{\textbf{B}(x,v_t)^c})(B)=k-1\big) d x.
	\end{align*}
	Hence, elementary arguments lead to
	\begin{align*}
		k\mathbb{P}(\xi_t(B)=k)
		& = M_t(B) \mathbb{P}\big(\xi_t(\eta_t|_{\textbf{B}(X,v_t)^c})(B)=k-1\big) 
		\\
		& = M_t(B)\mathbb{P}\big( \xi_t(B) + Z_t(u) =k-1\big),
	\end{align*}
	which is the desired conclusion.
\end{proof}

\begin{proof}[Proof of Theorem \ref{Thm...}]
Suppose $u > - \log(t)$ and let $Z_t(u)$ be as in Proposition \ref{prp.Sec-3.3}. We can rewrite $Z_t(u)$ as
	\begin{align*}
	Z_{t}(u) & =  \xi_t(\eta_t|_{\textbf{B}(X,v_t)^c})((u,\infty)) - \xi_t((u,\infty)) \\
	& = \sum_{z\in \eta_t\cap W \cap \mathbf{B}(X,2v_t) \cap \mathbf{B}(X,v_t)^c } \mathbf{1}\{ h_t(z,\eta_t|_{\mathbf{B}(X,v_t)^c})>u\} - \mathbf{1}\{h_t(z,\eta_t)>u\} \\
	& \quad - \sum_{z\in \eta_t\cap \mathbf{B}(X,v_t) \cap W} \mathbf{1}\{ h_t(z,\eta_t)>u\} \\
	& =: Z_{t,X}'(u) - Z_{t,X}''(u),
	\end{align*} 	
	where $Z_{t,X}'(u)$ and $Z_{t,X}''(u)$ are non-negative. For a fixed $x\in W$, the Mecke formula and short computations yield
		\begin{equation}\label{eq: expecZinr-1}\begin{split}
		\mathbb{E}[Z_{t,X}'(u)]	& \leq \mathbb{E}\Big[ \sum_{z\in\eta_t  \cap \textbf{B}(x,2v_t)\cap \textbf{B}(x,v_t)^c} \mathbf{1}\big\{h_t (z,\eta_t|_{\textbf{B}(x,v_t)^c})> u\big\} \Big] 
			\\
			& = t\int_{\textbf{B}(x,2v_t)\cap \textbf{B}(x,v_t)^c} \mathbb{P}\big( \eta_t(\textbf{B}(z, v_t)\cap \textbf{B}(x,v_t)^c )=0\big) dz
			\\
			& \leq t \int_{\textbf{B}(x,2v_t)\cap \textbf{B}(x,v_t)^c}  e^{-t v_t^d k_d/2}dz 
			\leq 2^{d}(u+\log(t))e^{- (u+ \log(t))/2 } 
			= 2^{d}\frac{u+\log(t)}{e^{u/2} \sqrt{t} }
	\end{split}\end{equation}
	and, similarly,
	\begin{equation}\label{eq: expecZinr-2}\begin{split} 
		\mathbb{E}[Z_{t,X}''(u)]	& \leq \mathbb{E}\Big[ \sum_{z\in\eta_t \cap  \textbf{B}(x,v_t)} \mathbf{1}\big\{h_t (z,\eta_t)> u \big\} \Big]
			=	  t\int_{\textbf{B}(x,v_t)} \mathbb{P}\big( \eta_t(\textbf{B}(z, v_t))=0\big) dz
			\\
			& \leq  t \int_{\textbf{B}(x,v_t)}  e^{-t v_t^d k_d}dz
			\leq (u+\log(t))e^{- u- \log(t) }
			=  \frac{u+\log(t)}{ e^{u} t}.
	\end{split}\end{equation}
	It follows from the triangle inequality that
	$$
	\mathbb{E}[|Z_{t}(u)|] 
	\leq 
	 2^{d}\frac{u+\log(t)}{e^{u/2} \sqrt{t} } + \frac{u+\log(t)}{ e^{u} t}.
	$$
	Then, by the first inequality of \eqref{fsxx} in Theorem \ref{Thmmm}, we obtain \eqref{ineq-th-1.3.3}.
	
	Let us now show \eqref{maxmintest}. We consider the cases $u\geq 0$, $u\in[-\log(\log(t)),0)$ and $u<-\log(\log(n))$ separately. Because of $u e^{-u}\leq 1$ and $u e^{-u/2}\leq 1$ for $u\geq 0$ and $\log(t)\geq 1$, by \eqref{ineq-th-1.3.3} we have
	$$
	d_{TV}\big(\xi_t((u,\infty)), P_{e^{-u}}\big)	\leq (2^{d+1}+2) \frac{\log(t)}{ \sqrt{t} } 
	$$
	for $u\geq 0$, which proves \eqref{maxmintest} for $u\geq 0$.
	
	In the following let $u\in [-\log(\log(t)),0)$ be fixed.   Since $Z_t(u)=Z_{t,X}'(u)-Z_{t,X}''(u)$ and the terms on the right-hand side are both non-negative, we obtain that
$$
Z_t(u)_+\leq Z_{t,X}'(u) \quad \text{and} \quad Z_t(u)_-\leq Z_{t,X}''(u).
$$
Combining these inequalities and \eqref{rrrf} in Theorem \ref{Thmmm} with $m=1$ establishes 
\begin{align*}
 &	\vert \mathbb{P}(T_t \leq u) - \mathbb{P}(P_{e^{-u}} =0) \vert =	\vert \mathbb{P}(\xi_t((u,\infty))=0) - \mathbb{P}(P_{e^{-u}} =0) \vert
	\\
	&\leq e^{u}\mathbb{E}[ \vert Z_t(u) \vert ] +  \mathbb{E}[ \vert Z_t(u) \vert \mathbf{1}\{ \xi_t((u,\infty))- Z_t(u)_{-} = 0 \} ] 
	\\
	&
	\leq e^u\mathbb{E}[ \vert Z_t(u) \vert ]+  \mathbb{E}[ Z_{t,X}'(u) \mathbf{1}\{ \xi_t((u,\infty))=0 \} ]+ \mathbb{E}[  Z_{t,X}''(u) ].
\end{align*}
Moreover,  by \eqref{eq: expecZinr-1} and \eqref{eq: expecZinr-2}  we have
\begin{equation*}
	\mathbb{E}[  Z_{t,X}'(u) ]\leq 2^{d}\frac{u+\log(t)}{e^{u/2} \sqrt{t} } \leq 2^{d} \frac{\log(t)}{e^{u/2} \sqrt{t} } \ \text{and} \ \mathbb{E}[  Z_{t,X}''(u) ]\leq \frac{u+\log(t)}{ e^{u} t}\leq  \frac{(\log(t))^2}{t} \leq \frac{\log(t)}{\sqrt{t}}.
\end{equation*}
Thus the identity $Z_t(u)=Z_{t,X}'(u)- Z_{t,X}''(u)$ with $Z_{t,X}'(u), Z_{t,X}''(u)\geq 0$ implies that
\begin{align}\label{eq:Neg-u-step1}
\vert \mathbb{P}(T_t \leq u) - \mathbb{P}(P_{e^{-u}} =0) \vert
\leq (2^{d} +2)\frac{\log(t)}{ \sqrt{t} }+  \mathbb{E}[  Z_{t,X}'(u) \mathbf{1}\{ \xi_t((u,\infty))=0 \} ].
\end{align}
 For $x\in W$ we define
$$
\xi_{t,x}((u,\infty))=\sum_{z\in\eta_t \cap W \cap  \textbf{B}(x, 4 v_t)^c} \mathbf{1}\{h_t (z,\eta_t)>u\}.
$$
Since, for $x\in W$,  $\mathbf{1}\{\xi_{t}((u,\infty))=0\}\leq \mathbf{1}\{\xi_{t,x}((u,\infty))=0\}$ and $Z_{t,x}'(u)$ and $\mathbf{1}\{\xi_{t,x}((u,\infty))=0\}$ are independent, we have
\begin{equation}\begin{split}\label{eq:bndZ=0-Inr}
	\mathbb{E}[Z_{t,X}'(u) \mathbf{1}\{\xi_t((u,\infty))=0\}] & \leq \int_W \mathbb{E}[Z_{t,x}'(u) \mathbf{1}\{\xi_{t,x}((u,\infty))=0\}]  dx \\
	& = \int_W \mathbb{E}[Z_{t,x}'(u)] \mathbb{P}(\xi_{t,x}((u,\infty))=0)  dx.
\end{split}\end{equation}
For $x\in W$, the Markov and the triangle inequalities, \eqref{ineq-th-1.3.3}  and  $e^{u/2}\sqrt{t}\geq 1$ imply that
\begin{align*}
	\mathbb{P}(\xi_{t,x}((u,\infty))=0) 
	& \leq \mathbb{P}(\xi_t((u,\infty))=0) + \mathbb{P}\Big(\sum_{z\in\eta_t \cap  \textbf{B}(x, 4 v_t)} \mathbf{1}\{h_t (z,\eta_t)>u\}>0\Big)
	\notag
	\\
	& \leq 2^{d} \frac{\log(t)}{ e^{u/2} \sqrt{t} } +  \frac{\log(t)}{e^{u} t} + e^{-e^{-u}} + \mathbb{E}\Big[\sum_{z\in\eta_t \cap  \textbf{B}(x, 4 v_t)} \mathbf{1}\big\{h_t (z,\eta_t) > u \big\} \Big]
	\notag
	\\
	& \leq (2^{d}+1) \frac{\log(t)}{e^{u/2} \sqrt{t} } + e^{-e^{-u}} + \mathbb{E}\Big[\sum_{z\in\eta_t \cap  \textbf{B}(x, 4 v_t)} \mathbf{1}\big\{h_t (z,\eta_t) > u \big\} \Big].
\end{align*}
Similar arguments as used in \eqref{eq: expecZinr-2} and $e^{u/2}\sqrt{t}\geq 1$ lead to
\begin{align*}
	\mathbb{E}\Big[\sum_{z\in\eta_t \cap  \textbf{B}(x, 4 v_t)} \mathbf{1}\big\{h_t (z,\eta_t) > u \big\} \Big]
	&\leq \frac{ 4^d (u+\log(t))}{e^{u} t} \leq \frac{ 4^d \log(t)}{e^{u/2}\sqrt{t}} .
\end{align*}
Since $\log(t) e^u\geq 1$ and $\frac{\log(t)^2}{\sqrt{t}} \leq 4$ for $t>e^2$, we obtain
$$
\frac{\log(t)}{e^{u/2} \sqrt{t} } \leq \frac{\log(t)^2 e^u}{e^{u/2} \sqrt{t} } \leq \frac{\log(t)^2}{\sqrt{t} } e^{u/2} \leq 4e^{u/2}.
$$
Together with $\exp(-e^{-u}- u/2) \leq 1,$ which follows from $u<0$, we have shown
$$
\mathbb{P}(\xi_{t,x}((u,\infty))=0) \leq 4( 4^d +2^d+1) e^{u/2} + e^{u/2} \leq (4( 4^d +2^d+1)+1) e^{u/2}
$$
so that, by \eqref{eq: expecZinr-1} and \eqref{eq:bndZ=0-Inr},
\begin{align*}
\mathbb{E}[Z_{t,X}'(u) \mathbf{1}\{\xi_t((u,\infty))=0\}]
& \leq (4( 4^d +2^d+1)+1) e^{u/2} \frac{2^{d}\log(t)}{e^{u/2}\sqrt{t}}
\\
& = (2^{d+2}( 4^d +2^d+1)+2^d) \frac{\log(t)}{\sqrt{t}}.
\end{align*}
Combining this with \eqref{eq:Neg-u-step1} leads to
\begin{equation}\label{gac}
	\begin{split}
		\left\vert \mathbb{P}(T_t
		\leq u ) - e^{-e^{-u}} \right\vert & \leq (2^{d+2}( 4^d +2^d+1)+2^d + 2^d+2) \frac{\log(t)}{\sqrt{t}} \\
		& \leq 2^{d+2}( 4^d +2^d+2) \frac{\log(t)}{\sqrt{t}},
	\end{split}
\end{equation}
which establishes \eqref{maxmintest} for $u\in[-\log(\log(t)),0)$.

	Finally for $u<-\log(\log(t))$ we have
	$$
	\mathbb{P}( T_t
	\leq u )
	\leq  \mathbb{P}( T_t
	\leq -\log(\log(t)) ),
	$$
	which by \eqref{gac} and the triangle inequality is bounded by 
	$$ 
	 	2^{d+2}( 4^d +2^d+2)\frac{\log(t)}{\sqrt{t}} + \frac{1}{t}.
	$$
	Therefore elementary arguments lead to
	\begin{align*}
		& \underset{u < - \log(\log(t)) }{\sup}\,\big\vert \mathbb{P} (T_t
		\leq u ) - e^{-e^{-u}} \big\vert 
		\leq [2^{d+2}( 4^d +2^d+2)+1]\frac{\log(t)}{\sqrt{t}},
	\end{align*}
which concludes the proof of \eqref{maxmintest}.
\end{proof}

\begin{remark}
Note that the integral in the middle of \eqref{eq: expecZinr-1} cannot be bounded with a better exponent for $t$. Indeed, using substitution, we can rewrite the integral as
$$
\frac{u + \log(t)}{k_d} \int_{\textbf{B}(0,2)\cap\textbf{B}(0,1)^c} e^{-(u + \log(t)) \frac{\lambda_d(\mathbf{B}(y,1)\cap\mathbf{B}(0,1)^c)}{k_d}} dy.
$$
For any sufficiently small $\varepsilon>0$ there exists a set $A\subset \textbf{B}(0,2)\cap\textbf{B}(0,1)^c$ with $\lambda_d(A)>0$ such that the ratio in the exponent is at least $(1+\varepsilon)/2$ for all $y\in A$. This provides a lower bound of the order $\log(t) t^{-(1+\varepsilon)/2}$.
\end{remark}

\section*{Acknowledgments}
This research was supported by the Swiss National Science Foundation (grant number 200021\_175584). We would like to thank Fraser Daly for some valuable comments.

\bibliography{bibliography}

\begin{thebibliography}{10}

\bibitem{MR972770}
R.~Arratia, L.~Goldstein, and L.~Gordon.
\newblock Two moments suffice for {P}oisson approximations: the {C}hen-{S}tein
  method.
\newblock {\em Ann. Probab.}, 17(1):9--25, 1989.

\bibitem{MR1092983}
R.~Arratia, L.~Goldstein, and L.~Gordon.
\newblock Poisson approximation and the {C}hen-{S}tein method.
\newblock {\em Statist. Sci.}, 5(4):403--434, 1990.

\bibitem{MR3896143}
R.~Arratia, L.~Goldstein, and F.~Kochman.
\newblock Size bias for one and all.
\newblock {\em Probab. Surv.}, 16:1--61, 2019.

\bibitem{MR1882476}
N.~Balakrishnan and M.~V. Koutras.
\newblock {\em Runs and scans with applications}.
\newblock Wiley Series in Probability and Statistics. Wiley-Interscience [John
  Wiley \& Sons], New York, 2002.

\bibitem{MR974580}
A.~D. Barbour.
\newblock Stein's method and {P}oisson process convergence.
\newblock {\em J. Appl. Probab.}, (25{\rm A}):175--184, 1988.

\bibitem{MR790624}
A.~D. Barbour and G.~K. Eagleson.
\newblock Poisson convergence for dissociated statistics.
\newblock {\em J. Roy. Statist. Soc. Ser. B}, 46(3):397--402, 1984.

\bibitem{MR1163825}
A.~D. Barbour, L.~Holst, and S.~Janson.
\newblock {\em Poisson approximation}, volume~2 of {\em Oxford Studies in
  Probability}.
\newblock The Clarendon Press, Oxford University Press, New York, 1992.

\bibitem{MR2274850}
A.~D. Barbour and A.~Xia.
\newblock On {S}tein's factors for {P}oisson approximation in {W}asserstein
  distance.
\newblock {\em Bernoulli}, 12(6):943--954, 2006.

\bibitem{MR2150719}
B.~B{\l}aszczyszyn and R.~Schott.
\newblock Approximations of functionals of some modulated-{P}oisson {V}oronoi
  tessellations with applications to modeling of communication networks.
\newblock {\em Japan J. Indust. Appl. Math.}, 22(2):179--204, 2005.

\bibitem{MR2654678}
P.~Calka.
\newblock Tessellations.
\newblock In {\em New perspectives in stochastic geometry}, pages 145--169.
  Oxford Univ. Press, Oxford, 2010.

\bibitem{MR3252817}
P.~Calka and N.~Chenavier.
\newblock Extreme values for characteristic radii of a {P}oisson-{V}oronoi
  tessellation.
\newblock {\em Extremes}, 17(3):359--385, 2014.

\bibitem{MR1632651}
L.~H.~Y. Chen.
\newblock Stein's method: some perspectives with applications.
\newblock In {\em Probability towards 2000 ({N}ew {Y}ork, 1995)}, volume 128 of
  {\em Lect. Notes Stat.}, pages 97--122. Springer, New York, 1998.

\bibitem{CHENAVIER20142917}
N.~Chenavier.
\newblock A general study of extremes of stationary tessellations with
  examples.
\newblock {\em Stochastic Processes and their Applications}, 124(9):2917--2953,
  2014.

\bibitem{MR3877881}
F.~Daly and O.~Johnson.
\newblock Relaxation of monotone coupling conditions: {P}oisson approximation
  and beyond.
\newblock {\em J. Appl. Probab.}, 55(3):742--759, 2018.

\bibitem{MR3502603}
L.~Decreusefond, M.~Schulte, and C.~Th\"{a}le.
\newblock Functional {P}oisson approximation in {K}antorovich-{R}ubinstein
  distance with applications to {U}-statistics and stochastic geometry.
\newblock {\em Ann. Probab.}, 44(3):2147--2197, 2016.

\bibitem{MR1133732}
A.~P. Godbole.
\newblock Poisson approximations for runs and patterns of rare events.
\newblock {\em Adv. in Appl. Probab.}, 23(4):851--865, 1991.

\bibitem{MR731691}
N.~Henze.
\newblock Ein asymptotischer {S}atz \"{u}ber den maximalen {M}inimalabstand von
  unabh\"{a}ngigen {Z}ufallsvektoren mit {A}nwendung auf einen {A}npassungstest
  im {${\bf R}^{p}$} und auf der {K}ugel.
\newblock {\em Metrika}, 30(4):245--259, 1983.

\bibitem{MR1876169}
O.~Kallenberg.
\newblock {\em Foundations of modern probability}.
\newblock Probability and its Applications (New York). Springer-Verlag, New
  York, second edition, 2002.

\bibitem{MR1472486}
V.~S. Koroljuk and Y.~V. Borovskich.
\newblock {\em Theory of {$U$}-statistics}, volume 273 of {\em Mathematics and
  its Applications}.
\newblock Kluwer Academic Publishers Group, Dordrecht, 1994.
\newblock Translated from the 1989 Russian original by P. V. Malyshev and D. V.
  Malyshev and revised by the authors.

\bibitem{MR1334894}
M.~V. Koutras, G.~K. Papadopoulos, and S.~G. Papastavridis.
\newblock Runs on a circle.
\newblock {\em J. Appl. Probab.}, 32(2):396--404, 1995.

\bibitem{MR2466550}
W.~Lao and M.~Mayer.
\newblock {$U$}-max-statistics.
\newblock {\em J. Multivariate Anal.}, 99(9):2039--2052, 2008.

\bibitem{MR3791470}
G.~Last and M.~Penrose.
\newblock {\em Lectures on the {P}oisson process}, volume~7 of {\em Institute
  of Mathematical Statistics Textbooks}.
\newblock Cambridge University Press, Cambridge, 2018.

\bibitem{MR1075417}
A.~J. Lee.
\newblock {\em {$U$}-statistics}, volume 110 of {\em Statistics: Textbooks and
  Monographs}.
\newblock Marcel Dekker, Inc., New York, 1990.

\bibitem{MR2394205}
M.~Mayer and I.~Molchanov.
\newblock Limit theorems for the diameter of a random sample in the unit ball.
\newblock {\em Extremes}, 10(3):129--150, 2007.

\bibitem{MR1295245}
J.~M{\o}ller.
\newblock {\em Lectures on random {V}orono\u{\i} tessellations}, volume~87 of
  {\em Lecture Notes in Statistics}.
\newblock Springer-Verlag, New York, 1994.

\bibitem{MR2933280}
S.~Y. Novak.
\newblock {\em Extreme value methods with applications to finance}, volume 122
  of {\em Monographs on Statistics and Applied Probability}.
\newblock CRC Press, Boca Raton, FL, 2012.

\bibitem{MR3992498}
S.~Y. Novak.
\newblock Poisson approximation.
\newblock {\em Probab. Surv.}, 16:228--276, 2019.

\bibitem{otto2020poisson}
M.~Otto.
\newblock Poisson approximation of {P}oisson-driven point processes and extreme
  values in stochastic geometry.
\newblock arXiv:2005.10116, 2020.

\bibitem{MR1986198}
M.~Penrose.
\newblock {\em Random geometric graphs}, volume~5 of {\em Oxford Studies in
  Probability}.
\newblock Oxford University Press, Oxford, 2003.

\bibitem{MR1442317}
M.~D. Penrose.
\newblock The longest edge of the random minimal spanning tree.
\newblock {\em Ann. Appl. Probab.}, 7(2):340--361, 1997.

\bibitem{MR3780386}
M.~D. Penrose.
\newblock Inhomogeneous random graphs, isolated vertices, and {P}oisson
  approximation.
\newblock {\em J. Appl. Probab.}, 55(1):112--136, 2018.

\bibitem{pianoforte2021criteria}
F.~Pianoforte and M.~Schulte.
\newblock Criteria for {P}oisson process convergence with applications to
  inhomogeneous {P}oisson-{V}oronoi tessellations.
\newblock arXiv:2101.07739, 2021.

\bibitem{POUPON2004233}
A.~Poupon.
\newblock Voronoi and {V}oronoi-related tessellations in studies of protein
  structure and interaction.
\newblock {\em Current Opinion in Structural Biology}, 14(2):233--241, 2004.

\bibitem{refId0}
{Ramella, M.}, {Boschin, W.}, {Fadda, D.}, and {Nonino, M.}
\newblock Finding galaxy clusters using {V}oronoi tessellations.
\newblock {\em A{\&}A}, 368(3):776--786, 2001.

\bibitem{MR2358635}
A.~R\"{o}llin.
\newblock Translated {P}oisson approximation using exchangeable pair couplings.
\newblock {\em Ann. Appl. Probab.}, 17(5-6):1596--1614, 2007.

\bibitem{MR2861132}
N.~Ross.
\newblock Fundamentals of {S}tein's method.
\newblock {\em Probab. Surv.}, 8:210--293, 2011.

\bibitem{MR2455326}
R.~Schneider and W.~Weil.
\newblock {\em Stochastic and integral geometry}.
\newblock Probability and its Applications (New York). Springer-Verlag, Berlin,
  2008.

\bibitem{MR2971726}
M.~Schulte and C.~Th\"{a}le.
\newblock The scaling limit of {P}oisson-driven order statistics with
  applications in geometric probability.
\newblock {\em Stochastic Process. Appl.}, 122(12):4096--4120, 2012.

\bibitem{MR3585403}
M.~Schulte and C.~Th\"{a}le.
\newblock Poisson point process convergence and extreme values in stochastic
  geometry.
\newblock In {\em Stochastic analysis for {P}oisson point processes}, volume~7
  of {\em Bocconi Springer Ser.}, pages 255--294. Bocconi Univ. Press, 2016.

\bibitem{MR511059}
B.~Silverman and T.~Brown.
\newblock Short distances, flat triangles and {P}oisson limits.
\newblock {\em J. Appl. Probab.}, 15(4):815--825, 1978.

\end{thebibliography}

\end{document}